\documentclass{amsart}
\usepackage{amsmath}
\usepackage{amssymb}
\usepackage{amsthm}
\usepackage{graphicx}
\usepackage{color}
\usepackage[all]{xypic}

\def\G{\dot {G}}
\def\U{\dot {U}}
\def\K{\dot {K}}

\def\T{\dot {T}}

\def\g{\mathfrak g}      
\def\h{\mathfrak h}      
\def\u{\mathfrak u}      
\def\t{\mathfrak t}      
\def\k{\mathfrak k}      
\def\n{\mathfrak n}      
\def\p{\mathfrak p}      
\def\a{\mathfrak a}
\def\l{\mathfrak l}
\def\b{\mathfrak b}

\def\c{\mathbf c}

\def\R{\mathbb R}        
\def\C{\mathbb C}        

\def\SU{\mathrm{SU}}     
\def\U{\mathrm{U}}       
\def\SL{\mathrm{SL}}     
\renewcommand{\sl}{\mathfrak{sl}}   

\newtheorem{theorem}{Theorem}[section]
\newtheorem{lemma}{Lemma}[section]

\newtheorem{proposition}{Proposition}[section]
\newtheorem{conjecture}{Conjecture}[section]
\theoremstyle{definition}
\newtheorem{definition}{Definition}

\newtheorem{example}{Example}[section]
\newtheorem*{remarks}{Remarks}

\theoremstyle{remark}
\newtheorem{remark}{Remark}[section]

\numberwithin{equation}{section}

\begin{document}
\title{Loops in Noncompact Groups of Inner Type and Factorization}
\author{Arlo Caine}
\email{jacaine@cpp.edu}

\author{Doug Pickrell}
\email{pickrell@math.arizona.edu}

\begin{abstract} In \cite{PP} we showed that a loop in a simply connected compact Lie group $\dot U$
has a unique Birkhoff (or triangular) factorization if and only if the loop has a
unique root subgroup factorization (relative to a choice of a reduced sequence of
simple reflections in the affine Weyl group). In this paper our main purpose is to investigate Birkhoff and
root subgroup factorization for loops in a noncompact type semisimple Lie group
$\dot G_0$ of inner type. In \cite{CP2} we showed that for an element of $\dot G_0$, i.e. a constant
loop, there is a unique Birkhoff factorization if and only if there is a root subgroup factorization.
However for loops in $\dot G_0$, while a root subgroup factorization
implies a unique Birkhoff factorization, there are several obstacles to
the converse. As in the compact case, root subgroup factorization is intimately related to
factorization of Toeplitz determinants.
\end{abstract}

\maketitle
\{2000 Mathematics Subject Classifications:  22E67\}

\setcounter{section}{-1}


\section{Introduction}

Finite dimensional Riemannian symmetric spaces come in dual pairs, one of compact
type and one of noncompact type. Given such a pair, there is a
diagram of finite dimensional groups
\begin{equation}\label{fdgroupdiagram}
\xymatrix{ & \dot G & \\ \dot G_0 \ar[ur] &  & \dot U \ar[ul] \\
 & \dot K \ar[ur] \ar[ul] & }
\end{equation}
where $\dot {U}$ is the universal covering of the identity component of the isometry group
of the compact type symmetric space $\dot {X}\simeq \dot {U}/\dot
{K}$, $\dot {G}$ is the complexification of $\dot {U}$, and
$\dot{G}_0$ is a covering of the isometry group for the dual noncompact
symmetric space $\dot {X}_0=\dot {G}_0/\dot {K}$.

The main purpose of this paper is to investigate Birkhoff (or triangular) factorization and ``root subgroup factorization" for the loop group of $\dot{G}_0$, assuming $\dot G_0$ is of inner type. Birkhoff factorization is investigated in \cite{CG} and \cite{PS}, from various points of view. In particular Birkhoff factorization for $L\dot U:=C^{\infty}(S^1,\dot U)$ is developed in Chapter 8 of \cite{PS}, using the Grassmannian model for the homogeneous space $L\dot U/\dot U$. Root subgroup factorization for generic loops in $\dot U$ appeared more recently in \cite{P3} (for $\dot U=SU(2)$, the rank one case) and in \cite{PP}. The Birkhoff decomposition for $L\G_0:=C^{\infty}(S^1,\dot G_0)$, i.e. the intersection of the Birkhoff decomposition for $L\dot G$ with $L\dot G_0$, is far more complicated than for $L\dot U$. With respect to root subgroup factorization, beyond loops in a torus (corresponding to imaginary roots), in the compact context the basic building blocks are exclusively spheres (corresponding to real roots), and in the inner noncompact context the building blocks are a combination of spheres and disks. This introduces additional analytic complications, and perhaps the main point of this paper is to communicate the problems that arise from noncompactness.

For $g\in L\dot U$, the basic fact is that $g$ has a unique triangular factorization if and only if $g$ has a unique ``root subgroup factorization" (relative to the choice of a reduced sequence of simple reflections in the affine Weyl group). This is
also true for elements of $\dot G_0$ (constant loops); see \cite{CP2}. However, somewhat to our surprise, this is far from true for loops in $\dot G_0$.

Relatively little sophistication is required to state the basic results, and identify the basic obstacles,
in the rank one noncompact case. This is essentially because (in addition to
loops in a torus) the basic building blocks are exclusively disks, and there is essentially a unique way to choose a reduced sequence of
simple reflections in the affine Weyl group, so that the dependence on this choice can be suppressed.

\subsection{The Rank 1 Case}

We consider the data determined by the Riemann sphere and the Poincar\'e disk. For this pair, the diagram (\ref{fdgroupdiagram}) becomes
\begin{equation}\label{rankonegroupdiagram}
\xymatrix{ & \SL(2,\mathbb C) & \\
\SU(1,1) \ar[ur] & & \SU(2) \ar[ul] \\
& \mathrm{S}(\U(1)\times \U(1)) \ar[ur] \ar[ul] & }
\end{equation}

Let $L_{fin}\SL(2,\mathbb C)$ denote the group consisting of maps $S^1\to SL(2,\mathbb C)$ having finite Fourier series, with pointwise multiplication. The subset of those functions having values in $\SU(1,1)$ is then a subgroup, denoted $L_{fin}\SU(1,1)$.

\begin{example}\label{noncompact_a} For each $\zeta \in \Delta:=\{\zeta\in\C\colon |\zeta|<1\}$ and $n\in \mathbb Z$, the function $S^1\to \SU(1,1)$ defined by
\begin{equation}\label{example0}
z \mapsto \mathbf a(\zeta) \begin{pmatrix} 1&
\zeta z^{-n}\\
\bar{\zeta}z^n&1\end{pmatrix}, \text{   where  }\mathbf a(\zeta)=(1-|\zeta|^2)^{-1/2},
\end{equation}
is in $L_{fin}\SU(1,1)$.
\end{example}
$L_{fin}\SU(2)$ and $L_{fin}\SU(1,1)$ are dense in the smooth loop groups $L\SU(2):=C^{\infty}(S^1,\SU(2))$ and $L\SU(1,1):=C^{\infty}(S^1,\SU(1,1))$, respectively. This is proven in the compact case in Proposition 3.5.3 of \cite{PS}, and the argument applies also for $\SU(1,1)$, taking into account the obvious modifications.

For a Laurent series $f(z)=\sum f_n z^n$, let $f^*(z)=\sum \bar f_n z^{-n}$. If $\Omega$ is a domain on the Riemann sphere, we write $H^0(\Omega)$ for the vector space of holomorphic scalar valued functions on $\Omega$.  If $f\in H^0(\Delta)$, then $f^* \in H^0(\Delta^*)$, where $\Delta^*$ denotes the open unit disk at $\infty$.

\begin{theorem}\label{SU(1,1)theorem1} Suppose that $g_1 \in L_{fin}\SU(1,1)$ and fix $n>0$. Consider the following three statements:

\begin{enumerate}
\item[(I.1)] $g_1$ is of the form
\[
g_1(z)=\begin{pmatrix} a(z)&b(z)\\
b^*(z)&a^*(z)\end{pmatrix},\quad z\in S^1,
\]
where $a$ and $b$ are polynomials in $z$ of order $n-1$ and $n$, respectively, with $a(0)>0$.
\item[(I.2)] $g_1$ has a ``root subgroup factorization'' of the form
\[
g_1(z)=\mathbf a(\eta_n)\begin{pmatrix} 1&\bar{\eta}_nz^n\\
\eta_nz^{-n}&1\end{pmatrix}..\mathbf
a(\eta_0)\begin{pmatrix} 1&
\bar{\eta}_0\\
\eta_0&1\end{pmatrix},
\]
for some sequence $(\eta_i)_{i=0}^n$ in $\Delta$ and $\mathbf{a}\colon \Delta\to \R$ is the function in Example \ref{noncompact_a}.
\item[(I.3)] $g_1$ has triangular factorization of the form
\[
\begin{pmatrix} 1&0\\
\sum_{j=0}^n \bar y_jz^{-j}&1\end{pmatrix}\begin{pmatrix} a_1&0\\
0&a_1^{-1}\end{pmatrix}\begin{pmatrix} \alpha_1 (z)&\beta_1 (z)\\
\gamma_1 (z)&\delta_1 (z)\end{pmatrix},
\]
where $a_1>0$, the third factor is a matrix valued polynomial in $z$ which is unipotent upper
triangular at $z=0$.
\end{enumerate}
Statements (I.1) and (I.3) are equivalent. (I.2) implies (I.1) and (I.3). If $g_1$ is in the identity connected component of the sets in (I.1) and (I.3), then the converse holds, i.e. $g_1$ has a root subgroup factorization as in (I.2).

There is a similar set of implications for $g_2 \in L_{fin}SU(1,1)$ and the following statements:
\begin{enumerate}
\item[(II.1)] $g_2$ is of the form
\[
g_2(z)=\begin{pmatrix} d^{*}(z)&c^{*}(z)\\
c(z)&d(z)\end{pmatrix},\quad z\in S^1,
\]
where $c$ and $d$ are polynomials in $z$ of order $n$ and $n-1$, respectively, with $c(0)=0$ and $d(0)>0$.
\item[(II.2)] $g_2$ has a ``root subgroup factorization'' of the form
\[
g_2(z)=\mathbf a(\zeta_n)\begin{pmatrix} 1&\zeta_nz^{-n}\\
\bar{\zeta}_nz^n&1\end{pmatrix}..\mathbf
a(\zeta_1)\begin{pmatrix} 1&
\zeta_1z^{-1}\\
\bar{\zeta}_1z&1\end{pmatrix},
\]
for some sequence $(\zeta_k)_{k=1}^n$ in $\Delta$ and $\mathbf{a}\colon \Delta\to \R$ is the function in Example \ref{noncompact_a}.
\item[(II.3)] $g_2$ has a triangular factorization of the form
\[
\begin{pmatrix} 1&\sum_{j=1}^n \bar x_jz^{-j}\\
0&1\end{pmatrix}\begin{pmatrix} a_2&0\\
0&a_2^{-1}\end{pmatrix}\begin{pmatrix} \alpha_2 (z)&\beta_2 (z)\\
\gamma_2 (z)&\delta_2 (z)\end{pmatrix},
\]
where $a_2>0$, and the third factor is a matrix valued polynomial in $z$ which is unipotent
upper triangular at $z=0$.
\end{enumerate}
When $g_1$ and $g_2$ have root subgroup factorizations, the scalar entries determining the diagonal factor have the product form
\begin{equation}
a_1=\prod_{i=0}^n \mathbf a(\eta_i) \text{ and } a_2^{-1}=\prod_{k=1}^n \mathbf
a(\zeta_k), \text{ respectively.}
\end{equation}
\end{theorem}

In general we do not know how to describe the connected component in the first and third conditions.
The following example shows how disconnectness arises in the simplest nontrivial case.

\begin{example} Consider the case $n=2$ and
$g_2$ is as in II.3 with $x=x_1z+x_2z^2$, $1-x_2\bar x_2\ne 0$,
$$\alpha_2=1-a_2^{-2}\bar x_1x_2z, \quad \beta_2=-\frac{a_2^{-2}\bar x_1^2 x_2}{1-x_2\bar x_2}$$
$$\gamma_2=\frac{x_1}{1-x_2\bar x_2}z+x_2z^2,\quad \delta_2=1+\frac{\bar x_1x_2}{1-x_2\bar x_2}z$$
and
$$a_2^2=\frac{(1-x_2\bar x_2)^2-x_1\bar x_1}{1-x_2\bar x_2}.$$
It is straightforward to check that this $g_2$ does indeed have values in $SU(1,1)$.
In order for $a_2^2>0$, there are two possibilities: the first is that both the numerator and denominator are positive, in
which case there is a root subgroup factorization, and the second is that both the top and bottom are negative, in which
case root subgroup factorization fails (because when there is a root subgroup factorization, $ \zeta_2= x_2$, and we must
have $\vert \zeta_2\vert<1$).
\end{example}

In order to formulate a general factorization result, we need a $C^{\infty}$ version of Theorem \ref{SU(1,1)theorem1}.

\begin{theorem}\label{SU(1,1)theorem1'} Suppose that $g_1 \in L\SU(1,1)$. The following conditions are equivalent:
\begin{enumerate}
\item[(I.1)] $g_1$ is of the form
\[
g_1(z)=\begin{pmatrix} a(z)&b(z)\\
b^*(z)&a^*(z)\end{pmatrix},\quad z\in S^1,
\]
where $a$ and $b$ are holomorphic in $\Delta$ and have $C^{\infty}$ boundary values, with $a(0)>0$.
\item[(I.3)] $g_1$ has triangular factorization of the form
\[
\begin{pmatrix} 1&0\\
y^*&1\end{pmatrix}\begin{pmatrix} a_1&0\\
0&a_1^{-1}\end{pmatrix}\begin{pmatrix} \alpha_1 (z)&\beta_1 (z)\\
\gamma_1 (z)&\delta_1 (z)\end{pmatrix},
\]
where $y$ is holomorphic in $\Delta$ with $C^{\infty}$ boundary values, $a_1>0$, and
the third factor is a matrix valued polynomial in $z$ which is unipotent upper
triangular at $z=0$.
\end{enumerate}
Similarly if $g_2 \in LSU(1,1)$, the following statements are equivalent:
\begin{enumerate}
\item[(II.1)] $g_2$ is of the form
\[
g_2(z)=\begin{pmatrix} d^{*}(z)&c^{*}(z)\\
c(z)&d(z)\end{pmatrix},\quad z\in S^1,
\]
where $c$ and $d$ are holomorphic in $\Delta$ and have $C^{\infty}$ boundary values, with $c(0)=0$ and $d(0)>0$.
\item[(II.3)] $g_2$ has a triangular factorization of the form
\[
\begin{pmatrix} 1&x^*\\
0&1\end{pmatrix}\begin{pmatrix} a_2&0\\
0&a_2^{-1}\end{pmatrix}\begin{pmatrix} \alpha_2 (z)&\beta_2 (z)\\
\gamma_2 (z)&\delta_2 (z)\end{pmatrix},
\]
where $a_2>0$, $x$ is holomorphic in $\Delta$ and has $C^{\infty}$ boundary values, $x(0)=0$, and the third factor is a matrix valued function which is holomorphic in $\Delta$ and has $C^{\infty}$ boundary values, and is unipotent
upper triangular at $z=0$.
\end{enumerate}
\end{theorem}

Let $\sigma\colon \SL(2,\C)\to \SL(2,\C)$ denote the anti-holomorphic involution of $\SL(2,\mathbb C)$ which fixes $SU(1,1)$; explicitly
\[
\sigma(\begin{pmatrix}a&b\\c&d\end{pmatrix})=\begin{pmatrix}d^*&c^*\\b^*&a^*\end{pmatrix}.
\]

The following theorem is the analogue of Theorem 0.2 of \cite{P3} (the notation
in part (b) is taken from Section 1 of \cite{P3}, and reviewed below the statement of the theorem).

\begin{theorem}\label{SU(1,1)theorem2}  Suppose $g\in LSU(1,1)_{(0)}$, the identity component. Then $g$ has a unique
``partial root subgroup factorization" of the form
\[
g(z)=\sigma(g_1^{-1}(z))\begin{pmatrix} e^{\chi(z)}&0\\
0&e^{-\chi(z)}\end{pmatrix}g_2(z)
\]
where $\chi \in C^{\infty}(S^1,i\mathbb R)/2\pi i\mathbb Z$ and
$g_1$ and $g_2$ are as in Theorem \ref{SU(1,1)theorem1'}, if and only $g$ has a triangular factorization $g=lmau$ (see (\ref{trifactorization}) below) such that the boundary values of $l_{21}/l_{11}$
and $u_{21}/u_{22}$ are $<1$ in magnitude on $S^1$.
\end{theorem}

The following example shows that the unaesthetic condition on the boundary values in part (b) is essential.

\begin{example} Consider $g_2$ as in Theorem \ref{SU(1,1)theorem1}. The loop $g=g_2^*$ (the Hermitian conjugate of $g_2$ around
the circle) has triangular factorization
\[
g=\begin{pmatrix} \alpha_2^*&\gamma_2^*\\
\beta_2^*&\delta_2^*\end{pmatrix}\begin{pmatrix} a_2&0\\
0&a_2^{-1}\end{pmatrix} \begin{pmatrix} 1&0\\
\sum_{j=1}^n x_jz^{j}&1\end{pmatrix}.
\]
If $n=2$, then $x_1=\bar{\zeta}_1(1-\vert\zeta_2\vert^2)$ and $x_2=\bar{\zeta}_2$, and this loop will often not satisfy the condition $\vert x_1z+x_2z^2\vert<1$ on $S^1$. In this case $g$ will not have a partial root subgroup factorization in the sense of Theorem \ref{SU(1,1)theorem2}.
\end{example}

The group $LSL(2,\mathbb C)$ has a Birkhoff decomposition
\[
LSL(2,\mathbb C)=\bigsqcup_{w\in W} \Sigma^{LSL(2,\mathbb C)}_w
\]
where $W$ (an affine Weyl group, and in this case the infinite dihedral group) is a quotient of a discrete group of unitary loops
\begin{eqnarray*}
W & = & \{\mathbf{w} =\begin{pmatrix} z^n&0\\0&z^{-n}\end{pmatrix}\left(\begin{matrix}0&-1\\1&0\end{matrix}\right)^{\epsilon}:n\in\mathbb Z,\epsilon\in \mathbb Z_4\}/\{\pm\left(\begin{matrix}1&0\\0&1\end{matrix}\right)\} \\
& = & \langle r_0,r_1 \vert r_0^2=r_1^2=1\rangle,
\end{eqnarray*}
where
\[
r_0=\left[\left(\begin{matrix}0&-z^{-1}\\z&0\end{matrix}\right)\right],
r_1=\left[\left(\begin{matrix}0&-1\\1&0\end{matrix}\right)\right]
\] (the reflections corresponding to the two simple roots for the Kac-Moody extension of $\sl(2,\mathbb C)$).
The set $\Sigma^{LSL(2,\mathbb C)}_w$ consists of loops which have a (Birkhoff) factorization of the form
\begin{equation}\label{trifactorization}g=l\cdot \mathbf{w} \cdot m\cdot a\cdot u,\end{equation}
where $w=[\mathbf w]$,
\[
l=\left(\begin{array}{cc}
l_{11}&l_{12}\\
l_{21}&l_{22}\end{array} \right)\in H^0(\Delta^{*},G),\quad
l(\infty )=\left(\begin{array}{cc}
1&0\\
l_{21}(\infty )&1\end{array} \right),
\]
$l$ has smooth boundary values on $S^1$, $m=\left(\begin{array}{cc}
m_0&0\\
0&m_0^{-1}\end{array} \right)$, $m_0\in S^1$,
$a=\left(\begin{array}{cc}
a_0&0\\
0&a_0^{-1}\end{array} \right)$, $a_0>0$,
\[u=\left(\begin{array}{cc}
u_{11}&u_{12}\\
u_{21}&u_{22}\end{array} \right)\in H^0(\Delta ,G),\quad
u(0)=\left(\begin{array}{cc}
1&u_{12}(0)\\
0&1\end{array} \right),\]
and $u$ has smooth boundary values on $S^1$. If $w=1$, the generic case, then we say (as in Section 1 of \cite{P3}) that $g$ has a triangular factorization, and in this case the factors are unique.

Next, let $LSU(1,1)_{(n)}$ denote the connected component containing
$\begin{pmatrix} z^n&0\\0&z^{-n}\end{pmatrix}\in Hom(S^1;\T)$, and let
\[
\Sigma^{LSU(1,1)}_w:=\Sigma^{LSL(2,\mathbb C)}_w \cap LSU(1,1)
\]
and
\[
\Sigma^{LSU(1,1)_{(n)}}_w:=\Sigma^{LSL(2,\mathbb C)}_w \cap LSU(1,1)_{(n)}.
\]
Based on finite dimensional results in \cite{CP2}, and the compact case, one might expect the following to be true:
\begin{enumerate}
\item[(1)] Modulo $\T$, the circle subgroup, it should be possible to contract $\Sigma^{LSU(1,1)_{(n)}}_w$ down to $w$;
in particular $\Sigma^{LSU(1,1)}_w$ should be empty unless $w$ is represented by a loop in $SU(1,1)$, e.g., $w=\left[\left(\begin{matrix}z^n&0\\0&z^{-n}\end{matrix}\right)\right]$ for some $n$.
\item[(2)] $\Sigma^{LSU(1,1)_{(0)}}_1=LSU(1,1)_{(0)}$.
\item[(3)] Each $\Sigma^{LSU(1,1)_{(n)}}_w$ should admit a relatively explicit parameterization.
\end{enumerate}
Statements (1) and (2) are definitely false; statement (3) is very elusive, if not doubtful.

\begin{proposition}$\phantom{a}$ \label{SU(1,1)theorem3}
\begin{enumerate}
\item[(a)] $\Sigma^{LSU(1,1)_{(n)}}_w$ nonempty does not imply that $w$ is represented by a loop in $SU(1,1)$.  For example, $\Sigma^{LSU(1,1)_{(1)}}_{r_1}$ is nonempty.
\item[(b)] $\Sigma^{LSU(1,1)_{(0)}}_1$ is properly contained in $LSU(1,1)_{(0)}$.
\end{enumerate}
\end{proposition}

To summarize, the set of loops having a root subgroup factorization is properly contained in the set of loops in the identity component which have a triangular factorization which, in turn, is a proper subset of the identity component of $LSU(1,1)$.  The first set has an explicit parametrization.  The second does not, a severe flaw, and the third is very simple topologically.  This stands in contrast to the compact case of $LSU(2)$ where there is only one connected component and every loop admitting a triangular factorization admitted a root subgroup factorization.

\subsection{Toeplitz Determinants}

The group $LSU(1,1)$ acts by bounded multiplication
operators on the Hilbert space $H:=L^2(S^1;\mathbb C^2)$. As in chapter 6 of
\cite{PS}, this defines a homomorphism of $LSU(1,1)$ into the
restricted general linear group of $H$ defined relative to the
Hardy polarization $H=H_+\oplus H_-$, where $H_+$ is the subspace
of boundary values of functions in $H^0(\Delta,\C^2)$ and $H_-$ is
the subspace of boundary values of functions in
$H^0(\Delta^*,\C^2)$. For a loop $g$, let $A(g)$ (respectively, $A_1(g)$) denote the
corresponding Toeplitz operator, i.e., the compression of
multiplication by $g$ to $H_+$ (resp., the shifted Toeplitz operator, i.e. the compression
to $H_+\ominus \mathbb C\left(\begin{matrix}0\\1\end{matrix}\right)$). It is well known that $A(g)A(g^{-1})$ and $A_1(g)A_1(g^{-1})$ are determinant class operators (i.e., of the form $1+\text{trace
class}$).

\begin{theorem}\label{Toeplitz} Suppose that $g\in
LSU(1,1)_{(0)}$ has a root subgroup factorization as in part (b) of Theorem \ref{SU(1,1)theorem2}.
Then
\[
\det(A(g)A(g^{-1}))=
\left(\prod_{i=0}^{\infty}\frac{1}{(1-\vert\eta_i\vert^2)^{i}}\right)\times
\left(\prod_{
j=1}^{\infty}e^{-2j\vert\mathbf{\chi}_j\vert^2}\right)\times
\left(\prod_{k=1}^{\infty}\frac
{1}{(1-\vert\zeta_k\vert^2)^{k}}\right),
\]
\[
\det(A_1(g)A_1(g^{-1}))=
\left(\prod_{i=0}^{\infty}\frac{1}{(1-\vert\eta_i\vert^2)^{i+1}}\right)\times
\left(\prod_{
j=1}^{\infty}e^{-2j\vert\mathbf{\chi}_j\vert^2}\right)\times
\left(\prod_{k=1}^{\infty}\frac
{1}{(1-\vert\zeta_k\vert^2)^{k-1}}\right)
\]
and if $g=lmau$ is the triangular factorization as in (\ref{trifactorization}) (with $\mathbf w=1$), then
\[
a_0^2=\frac{\det(A_1(g)A_1(g^{-1}))}{\det(A(g)A(g^{-1}))}=\frac{\prod_{k=1}^{\infty}(1-\vert\zeta_k\vert^2)}
{\prod_{i=0}^{\infty}(1-\vert\eta_i\vert^2)}.
\]
\end{theorem}

When $(\eta_i)_{i=0}^\infty$ and $(\zeta_k)_{k=1}^\infty$ are the
zero sequences (the abelian case), the first formula specializes to a result of
Szego and Widom (see Theorem 7.1 of \cite{W2}). Estelle Basor
pointed out to us that this result, for $g$ as in (\ref{example0}), can be deduced from Theorem 5.1 of \cite{W1}.

\subsection{Additional Motivation}

There is potentially a useful analogue of root subgroup factorization for the group of homeomorphisms of a circle
(see \cite{P4}). This is a very complicated (negatively curved type) group, and we are currently stumbling around in
trying to understand what can go wrong. In this paper our primary contribution is perhaps to identify what can go wrong with
Birkhoff and root subgroup factorization for loops into a noncompact target. There are other variations on root subgroup
factorization as well (see \cite{BP}). 

From another point of view, it is expected that root subgroup factorization is relevant to finding Darboux coordinates for homogeneous Poisson structures on $L\dot U$ and $L\G_0$ (see \cite{Pi1}). As of this writing, this is an open question.

\subsection{Plan of the Paper} This paper is essentially a sequel to \cite{PP} and \cite{CP2}. We will refer to
the latter paper as the `finite dimensional case', and we note the differences as we go along.

Section \ref{background} is on background for finite dimensional groups (which is identical to \cite{CP2}) and loop groups.
In section \ref{noncpt1} we consider the intersection of the Birkhoff decomposition for $L\G$
with $L\G_0$. Unfortunately for loops in $\dot G_0$, there does not exist an analogue of ``block (coarse) triangular decomposition", a key feature of the finite dimensional case. Consequently there does not exist a reduction to the compact type case, as in finite dimensions. One might still naively expect that there could be
a relatively transparent way to parameterize the Birkhoff components intersected with $L\G_0$ (as in the finite dimensional case, and in the case of loops into compact groups, e.g. using root subgroup factorization). But these intersections turn out to not be so simple topologically. Most of the section is devoted to rank one examples which illustrate the basic complications.

In Section \ref{topstratum} we consider root subgroup factorization for generic loops in $\G_0$.
Our objective in this section is to prove partial analogues of Theorems 4.1, 4.2, and 5.1 of \cite{PP}, for
generic loops in the identity component of (the Kac-Moody central extension of) $L\G_0$ (when $\G_0$ is of inner type).
As in the rank one case above, all of the statements have to be severely modified. The structures of the arguments in this noncompact context are roughly the same as in \cite{PP}, but there many differences in the details (reflected in the more
complicated statements of theorems).

\subsection{Acknowledgement}

The second author thanks Hermann Flaschka, whose questions motivated us to consider loops in noncompact groups. We also thank Estelle Basor for many useful conversations.  The first author thanks the Provost's Teacher-Scholar Program at California State Polytechnic University Pomona for supporting this work.

\section{Notation and Background}\label{background}

Throughout we fix a real simple noncompact type Lie algebra which is inner, denoted $\dot \g_0$, and we let $\dot\g:=\dot \g_0^{\mathbb C}$.

\subsection{Data determined by the choice of a Cartan involution}A choice of Cartan involution $\Theta$ on $\dot\g_0$ determines a maximal compact type Lie subalgebra $\dot\k$ of $\dot \g_0$.  Let $\sigma$ denote the canonical complex conjugation on $\dot\g$ fixing $\dot\g_0$.  If we extend $\Theta$ to $\dot\g$ complex linearly, then the composition $\tau=\sigma\circ \Theta$ is a complex conjugation on $\dot\g$ fixing a compact real form $\dot\u$ of $\dot\g$.  The extended involution $\Theta$ on $\dot\g$ stabilizes $\dot\u$ and fixes $\dot\k$ inside of $\dot\u$.  Thus, $\dot\g_0\cap\dot\u=\dot\k$.  The assumption that $\g_0$ is of inner type is equivalent to the condition that
\begin{equation}\label{rank_condition}
\mathrm{rank}(\dot\g_0)=\mathrm{rank}(\dot{\k})=\mathrm{rank}(\dot\u).
\end{equation}

We write $\dot\g_0=\dot\k+\dot\p$ for the decomposition of $\dot\g_0$ into the eigenspaces of $\Theta$ on $\dot\g_0$.  Then $\dot\u=\dot\k+i\dot\p$ where multiplication by $i$ denotes the canonical complex structure on $\dot\g$, and this is the decomposition of $\dot\u$ into the eigenspaces of the extension of $\Theta$ restricted to $\dot\u$.
We will also use $\Theta$ to denote the corresponding holomorphic involution of $\dot G$, and its restrictions to $\dot G_0$ and $\dot U$, which fixes $\dot K$ in $\dot G_0$ and $\dot U$, respectively.

\subsection{Data determined by the choice of a $\Theta$-stable Cartan subalgebra and a Weyl chamber in the Inner Case}

Fix a Cartan subalgebra $\dot{\t}\subset \dot{\k}$.
Because of our rank assumption (\ref{rank_condition}), $\dot{\t}$ is a
$\Theta$-stable Cartan subalgebra of $\dot{\g}_0$ and every $\Theta$-stable Cartan subalgebra of $\dot\g_0$ is of this form.  In addition, $\dot\t$ is a
$\Theta$-stable Cartan subalgebra of $\dot{\u}$, and its
centralizer $\dot{\h}$ in $\dot{\g}$ is a $\Theta$-stable Cartan
subalgebra of $\dot{\g}$.  We write $\dot{\h}=\dot{\t}+\dot\a$, where $\dot\a=i\t$,
for the eigenspace decomposition of $\dot\h$ under $\Theta$ and let $\dot H=\exp(\dot\h)$, $\dot T=\exp(\dot\t)$, and $\dot A=\exp(\dot\a)$, respectively.

We will use $\dot W:=N_{\dot U}(\dot T)/\dot T$ as a model for the Weyl group of $(\dot\g,\dot\h)$.  The choice of a Weyl chamber $C$ in $\dot\a$ determines a choice of positive roots for the action of $\dot\h$ on $\dot\g$. Let $\dot\n^+$ denote the sum of the root spaces indexed by positive roots and $\dot\n^-$ denote the sum of the root spaces indexed by negative roots.  In this way, the choice of a $\Theta$-stable Cartan subalgebra $\dot\t$ of $\dot\g_0$ and a Weyl chamber $C$ determines a triangular decomposition
\begin{equation}\label{triangledecomp}
\dot{\g}=\dot{\n}^{-}+ \dot{\h} + \dot{\n}^{+}.
\end{equation}

Set $\dot N^\pm=\exp(\n^\pm)$.  Then $\dot B^+=\dot H\dot N^+$ and $\dot B^-=\dot N^-\dot H$ are a pair of opposite Borel subgroups of $\dot G$. A consequence of the stability of $\dot\h$ under $\sigma$ and $\tau$ is the fact that $\sigma(\dot\n^\pm)=\dot\n^\mp$ and $\tau(\dot\n^\pm)=\dot\n^\mp$.

\begin{example}\label{su(1,1)conventions}
In this paper, a special role is played by the rank 1 example of
\[
\mathfrak{su}(1,1)=\left\{\begin{pmatrix}iz & x+iy \\ x-iy & -iz\end{pmatrix}\colon x,y,z\in\R\right\}
\]
with Cartan involution given by
\[
\mathrm{Ad}_g\text{ where }g=\begin{pmatrix} i & 0 \\ 0 & -i\end{pmatrix}.
\]
The effect of this involution is to negate the off-diagonal entries.  In this case, the maximal compact subalgebra fixed by the involution is the one dimensional subalgebra $\mathfrak{s}(\mathfrak{u}(1)\times \mathfrak{u}(1))$ of diagonal matrices in $\mathfrak{su}(1,1)$, which is abelian and hence a Cartan subalgebra.  The complexification is $\mathfrak{sl}(2,\C)$ and the associated compact real form is$\mathfrak{su}(2)$.  The involution fixing $\mathfrak{su}(2)$ is $X\mapsto -X^*$ (opposite conjugate transpose).  We will work with the standard triangular decomposition
\begin{equation}\label{std_tri_decomp_sl(2,C)}
\mathfrak{sl}(2,\C)=\mathrm{span}_\C\left\{\begin{pmatrix} 0& 0 \\ 1 & 0 \end{pmatrix}\right\}+\mathrm{span}_\C\left\{\begin{pmatrix} 1 & 0 \\ 0 & -1 \end{pmatrix}\right\}+\mathrm{span}_\C\left\{\begin{pmatrix} 0 & 1 \\ 0 & 0 \end{pmatrix}\right\}.
\end{equation}
\end{example}

\subsection{Root Data}
Let $\dot\theta$ denote the highest root and normalize the Killing form so that (for the dual form) $\langle \dot\theta,\dot\theta\rangle=2$.  For each root $\dot\alpha$ let $h_{\dot\alpha}\in\a$ denote the associated coroot.
The inner type assumption, together with the $\Theta$-stability of $\dot\h$, implies that each root space $\dot\g_{\dot\alpha}$ is contained in either $\dot\k^\C$ or in $\dot\p^\C$ and thus the roots can be sorted into two types.  A root $\dot\alpha$ is of \emph{compact type} if the root space $\dot\g_{\dot\alpha}$ is a subset of $\dot\k^\C\subset\dot\g$ and of \emph{noncompact type} otherwise, i.e., when $\dot\g_{\dot\alpha}\subset\dot\p^\C$.  The following is elementary.

\begin{proposition}
For each simple positive root $\dot\gamma$ there exists a Lie algebra homomorphism $\iota_{\dot\gamma}\colon \mathfrak{sl}(2,\C)\to\dot\g$ which carries the standard triangular decomposition of $\mathfrak{sl}(2,\C)$ (\ref{std_tri_decomp_sl(2,C)}) into the triangular decomposition $\dot\g=\dot\n^-+\dot\h+\dot\n^+$ and:
\begin{enumerate}
\item[(a)] in any case $\iota_{\dot\gamma}$ restricts to a homomorphism $\iota_{\dot\gamma}\colon \mathfrak{su}(2)\to \dot\u$;
\item[(b)] when $\dot\gamma$ is of compact type then $\iota_{\dot\gamma}$ restricts to  $\iota_{\dot\gamma}\colon \mathfrak{su}(2)\to\dot\k$;
\item[(c)] when $\dot\gamma$ is of noncompact type then $\iota_{\dot\gamma}$ restricts to $\iota_{\dot\gamma}\colon \mathfrak{su}(1,1)\to \dot\g_0$.
\end{enumerate}
\end{proposition}

We denote the corresponding group homomorphism by the same symbol.  Note that if $\dot\gamma$ is of noncompact type, then $\iota_{\dot\gamma}$ induces an embedding of the rank one diagram (\ref{rankonegroupdiagram}) into the finite dimensional group diagram (\ref{fdgroupdiagram}).  For each simple positive root $\dot\gamma$, we use the group homomorphism to set
\begin{equation}\label{defn_of_r_gamma}
\mathbf r_{\dot\gamma}=\iota_{\dot\gamma}\begin{pmatrix} 0 & i \\ i & 0\end{pmatrix}\in N_{\dot U}(\dot T)
\end{equation}
and obtain a specific representative for the associated simple reflection $r_{\dot\gamma}\in W=N_{\dot U}(\dot T)/\dot T$ corresponding to $\dot\gamma$.  (We will adhere to the convention of using boldface letters to denote representatives of Weyl group elements).

The coroots $h_{\dot\alpha_1},h_{\dot\alpha_2},\dots,h_{\dot\alpha_r}$ form a basis for
$\dot\a$ and the dual basis consists of the fundamental
weights $\dot{\Lambda}_1,\dot{\Lambda}_,\dots, \dot{\Lambda}_r$.
Introduce the lattices
\[
\widehat{\dot T}=\bigoplus_{1\le i\le r} \mathbb Z \dot{\Lambda}_i
\quad \text{(weight lattice)}, \quad \text{and} \quad \check{\dot
T}=\bigoplus_{1\le i\le r} \mathbb Z \dot{h}_i \qquad
\text{(coroot lattice)}.
\]

Recall that the kernel of $\exp:\dot{\mathfrak t} \to \dot T$ is $2 \pi i\check{\dot{T}}$.
Consequently there is a natural identification $\widehat{\dot T} \simeq \mathrm{Hom}(\dot T,\mathbb T)$,
where a weight $\dot{\Lambda}$ corresponds to the character $\exp(2\pi ix) \to \exp(2\pi
i\dot{\Lambda} (x))$, for $x\in \dot\a$.
Likewise, there is a natural identification $\check{\dot T} \to \mathrm{Hom}(\mathbb T,\dot T)$,
where an element $h$ of the coroot
lattice corresponds to the homomorphism $\mathbb T\to \dot
T$ given by $\exp(2\pi ix)\mapsto \exp(2\pi ixh)$, for $x\in \mathbb R$.

In addition, let
\[
\widehat{\dot{R}}=\bigoplus_{1\le i\le r} \mathbb Z \dot{\alpha}_i \quad \text{(root lattice)}, \quad
\text{and} \quad \check{\dot R}=\bigoplus_{1\le i\le r} \mathbb Z
\dot{\Theta}_i \quad \text{(coweight lattice)}.
\]
Then these lattices and bases are also in duality. The elements $\dot{\Theta}_1,\dots,\dot{\Theta}_r$ are the fundamental coweights.

The affine Weyl group for $\dot{\g}$ is the semidirect product $\dot W \propto
\check{\dot T}$. For the action of $\dot W$ on $\dot{\mathfrak
h}_{\mathbb R}$, a fundamental domain is the positive Weyl chamber $C$.  For the natural affine
action
\begin{equation}\label{affineaction}
(\dot W \propto \check{\dot T})\times \dot\a \to \dot\a
\end{equation}
a fundamental domain is the convex set
\[
C_0=\{x\in C:\dot{\theta}(x)<1\} \qquad
\text{(fundamental alcove)}.
\]

\subsection{Affine Lie Algebras}

Let $L\dot{\mathfrak g}=C^{\infty}(S^1,\dot{\mathfrak g})$, viewed
as a Lie algebra with pointwise bracket. There is a universal
central extension
\[0\to \mathbb C c\to\widetilde {L}\dot {\mathfrak g}\to L\dot {\mathfrak g}\to 0,\]
where $\widetilde {L}\dot {\mathfrak g}=L\dot
{\mathfrak g}\oplus \mathbb C c$ as a vector space, and in these coordinates
\begin{equation}\label{bracket}[X+\lambda c,Y+\lambda^{\prime} c]_{\widetilde {L}\dot {\mathfrak g}}
=[X,Y]_{L\dot {\mathfrak g}}+\frac i{2\pi}\int_{S^1}\langle
X\wedge dY\rangle c. \end{equation} The smooth completion of the
untwisted affine Kac-Moody Lie algebra corresponding to
$\dot{\mathfrak g}$ is
$$\widehat L\dot{\mathfrak g}=\mathbb C d\propto\widetilde {L}\dot {\mathfrak g}
\qquad \text{(the semidirect sum)},$$ where the derivation $d$
acts by $d(X+\lambda c)=\frac 1i\frac d{d\theta}X$, for $X\in
L\dot {\mathfrak g}$, and $[d,c]=0$.

\begin{proposition}\label{unitaryext} For both $L\dot{\mathfrak u}$ and $L\dot {\mathfrak
g_0}$, the cocycle  $(X,Y)\to \int_{S^1}\langle X\wedge dY\rangle$
is real-valued. In particular the affine extension induces a
unitary central extension
\[0\to i\mathbb R c\to\widetilde {L}\dot {\mathfrak g_0}\to L\dot {\mathfrak g_0}\to 0 \]
and a real form $\widehat L\dot{\mathfrak g_0}=i\mathbb R
d\propto\widetilde {L}\dot {\mathfrak g_0}$ for $\widehat
L\dot{\mathfrak g}$ (and similarly for unitary loops as in
\cite{PP}).
\end{proposition}

We identify $\dot {\mathfrak g}$ with the constant loops in $L\dot
{\mathfrak g}$. Because the extension is trivial over $\dot
{\mathfrak g}$, there are embeddings of Lie algebras
\[\dot {\mathfrak g}\to\widetilde {L}\dot {\mathfrak g}\to \widehat L\dot{\mathfrak g}. \]

There are triangular decompositions
\begin{equation}\label{looptriangulardecomposition}\widetilde L\dot{\mathfrak
g} =\mathfrak n^{-}+ \mathfrak h +\mathfrak n^{ +}
\quad \text{and} \quad \widehat L\dot{\mathfrak g} =\mathfrak
n^{-}+ (\mathbb C d+\mathfrak h) +\mathfrak n^{ +},
\end{equation} where $\mathfrak h=\dot {\mathfrak h}+\mathbb C c$
and $\mathfrak n^{\pm}$ is the smooth completion of
$\dot{\mathfrak n}^{\pm} +\dot{\mathfrak g}(z^{\pm 1}\mathbb
C[z^{\pm 1}])$, respectively. The simple roots for $(\widehat
L_{fin}\dot{\mathfrak g},\mathbb C d+\mathfrak h)$ are
$\{\alpha_j: 0\le j\le r\}$, where
\[\alpha_0=d^*-\dot{\theta} ,\quad\alpha_j=\dot{\alpha}_j,\quad j>0, \]
$d^*(d)=1$, $d^*(c)=0$, $d^* (\dot {\mathfrak h})=0$, and the
$\dot{\alpha}_j$ are extended to $\mathbb C d+\mathfrak h$ by
requiring $\dot{\alpha}_ j(c)=\dot{\alpha}_ j(d)=0$. The simple
coroots are $\{h_j:0\le j\le r\}$, where
\[h_0=c-\dot {h}_{\dot{\theta}},\quad h_j=\dot {h}_j,\quad j>0.\]
For $i>0$, the root homomorphism $\iota_{\alpha_i}$ is
$\iota_{\dot{\alpha}_i}$ followed by the inclusion $\dot {\mathfrak
g}\subset \widetilde L \dot{\mathfrak g}$. For $i=0$
\begin{equation}\label{roothom}
\iota_{\alpha_0}(\begin{pmatrix} 0&0\\
1&0\end{pmatrix})=e_{\dot{\theta}}z^{-1},\quad
\iota_{\alpha_0}(\begin{pmatrix}
0&1\\
0&0\end{pmatrix})=f_{\dot{\theta}}z,
\end{equation}
where
$\{f_{\dot{\theta}},\dot {h}_{\dot{\theta}},e_{\dot{\theta}}\}$
satisfy the $\sl(2,\mathbb C )$-commutation relations, and
$e_{\dot{\theta}}$ is a highest root vector for $\dot {\mathfrak
g}$. The fundamental dominant integral functionals on $\mathfrak
h$ are $\Lambda_j$, $j=0,..,r$.

We set $\mathfrak t=i\mathbb R c \oplus \dot{\mathfrak t}$ and
$\mathfrak a=\mathfrak h_{\mathbb R}=\mathbb R c \oplus
\dot{\mathfrak h}_{\mathbb R}$.

\subsection{Loop Groups and Extensions}\label{loopgroups}

Let $\Pi:\widetilde L \dot G \to L \dot G$ ($\Pi:\widetilde L \dot
G_0 \to L \dot G_0$) denote the universal central $\mathbb C^*$
(resp., $\mathbb T$) extension of the smooth loop group $L \dot G$ (resp. $L
\dot G_0$).

\begin{proposition} $\Pi$ induces a central circle extension
$$1\to \mathbb T \to \widetilde L \dot G_0 \to L \dot G_0 \to 1$$
(and similarly for unitary loops as in \cite{PP}).
\end{proposition}

\begin{proof} This follows from Proposition \ref{unitaryext}.
\end{proof}

Let $N^{\pm}$ denote the subgroups
corresponding to $\mathfrak n^{\pm}$. Since the restriction of
$\Pi$ to $N^{\pm}$ is an isomorphism, we will always identify
$N^{\pm}$ with its image, e.g. $l\in N^+$ is identified with a
smooth loop having a holomorphic extension to $\Delta$ satisfying
$l(0)\in \dot N^+$. Also, set $T=\exp(\mathfrak t)$ and
$A=\exp(\mathfrak a)$.

As in the finite dimensional case, for $\widetilde {g}\in
N^{-}\cdot T A\cdot N^{+}\subset\widetilde {L}\dot G$, there is a
unique triangular decomposition
\begin{equation}\label{diagonal}\widetilde {g}=l\cdot d \cdot u,\quad \text{where}\quad d
=ma=\prod_{j=0}^r\sigma_j(\widetilde { g})^{h_j},\end{equation}
and $\sigma_j=\sigma_{\Lambda_j}$ is the fundamental matrix
coefficient for the highest weight vector corresponding to
$\Lambda_j$. If $\Pi(\widetilde {g})=g$, then because
$\sigma_0^{h_0}=\sigma_0^{c-\dot {h}_{\dot{\theta}}}$ projects to
$\sigma_ 0^{-\dot {h}_{\dot{\theta}}}$, $g=l\cdot \Pi(d)\cdot u$,
where
\begin{equation}\label{loopdiagonal}\Pi(d)(g)=
\sigma_0(\widetilde {g})^{-\dot
{h}_{\dot{\theta}}}\prod_{j=1}^r\sigma_j(\widetilde {g})^{\dot
{h}_j}=\prod_{j=1}^r\left (\frac {\sigma_j(\widetilde
{g})}{\sigma_0(\widetilde {g})^{\check {a}_j}}\right )^{\dot
{h}_j},
\end{equation} and the $\check {a}_j$ are positive integers such
that $\dot {h}_{\dot{\theta}}=\sum\check {a}_j\dot {h}_j$ (these
numbers are also compiled in Section 1.1 of \cite{KW}).

\begin{proposition}\label{notationlemma}$\phantom{a}$
\begin{enumerate}
\item[(a)] $N^{\pm}$ are stable with respect to $\Theta$, whereas
$N^{\pm}$ are interchanged by $(\cdot)^{*}$. If $\tilde{g}$ has
triangular factorization $\tilde{g}=l\cdot
m(\tilde{g})a(\tilde{g})\cdot u$ as in (\ref{diagonal}), then
$$\Theta(\tilde{g})=\Theta(l)\cdot m(\tilde{g})a(\tilde{g})\cdot
\Theta(u)$$ and
$$\tilde{g}^*=u^*\cdot m(\tilde{g})^*a(\tilde{g})\cdot
l^*$$ are triangular factorizations.

\item[(b)] If $\widetilde{g} \in \widetilde{L}\dot G$, then
$\sigma_j(\Theta(\tilde{g}))=\sigma_j(\tilde{g})$ and
$\sigma_j(\tilde{g}^*)=\sigma_j(\tilde{g})^*$.

\item[(c)] If $\widetilde{g} \in \widetilde{L}\dot G_0$, then
$\vert\sigma_j(\widetilde{g})\vert$ depends only on
$g=\Pi(\widetilde g)$, and
\begin{equation}\label{formula_for_sigma_j}
\vert\sigma_j\vert(g):=\vert\sigma_j(\widetilde{g})\vert
=(\sigma_j(\widetilde{g})\sigma_j(\widetilde{g}^{-1}))^{1/2}.
\end{equation}

\item[(d)] For $\widetilde g\in \widetilde L\dot G_0$ and
$g=\Pi(\widetilde g)$, $\widetilde g$ has a triangular
factorization if and only if $g$ has a triangular factorization.
The restriction of the projection $\widetilde{L}\dot G_0 \to L\dot
G_0$ to elements with $m(\widetilde g)=1$ is injective.
\end{enumerate}
\end{proposition}

\begin{proof} (a) and (b) follow from the compatibility of the
triangular factorization with respect to $\Theta$ and $\u$.

The first part of (c) follows from the fact that the induced
extension $\widetilde{L}\dot G_0$ is unitary. The formula \ref{formula_for_sigma_j}
in (c) follows from the fact that if $\lambda\in \mathbb T$ ,then
$$\sigma_j(\widetilde{g}\lambda)=\lambda^l\sigma_j(\widetilde{g})$$
where $l$ is the level.
\end{proof}

\subsection{A Note on the Rank One Case}

In this subsection we will freely use the notation in Section 1 of
\cite{P3} and \cite{PS} (as in section 1 of \cite{P3}, we denote
the Toeplitz and shifted Toeplitz operators by $A$ and $A_1$,
respectively).

In the rank one case $\sigma_0$ and $\sigma_1$ can be concretely
realized as ``regularized Toeplitz determinants." In the notation
of section 6.6 of \cite{PS}, a concrete model for the central
extension is
$$\widetilde{L}\dot
G=\{[g,q]:(g,q)\in L\dot G\times GL(H_+), A(g)q^{-1}=1+\text{trace
class}\}$$ (here $\dot{G}=SL(2,\mathbb C)$, $H=L^2(S^1,\mathbb
C^2)$, and $H_+$ is the subspace of boundary values of holomorphic
functions on the disk). In this realization
$$\sigma_0([g,q])=\det(A(g)q^{-1})$$

\begin{proposition}For $g\in L\dot{G}_0$, using the notation in Proposition \ref{notationlemma},
$$\vert\sigma_0\vert^2(g)=\det(A(g)A(g^{-1}))\text{ and }\vert\sigma_1\vert^2(g)\det(A_1(g)A_1(g^{-1})).$$
\end{proposition}

\begin{proof}This follows from (c) of Proposition \ref{notationlemma}. \end{proof}

\subsection{Reduced Sequences in the Affine Weyl
Group}\label{Weylgroup}

The Weyl group $W$ for $(\widehat L\dot{\mathfrak g},\mathbb
Cd+\mathfrak h)$ acts by isometries of $(\mathbb R d+\mathfrak
h_{\mathbb R},\langle \cdot,\cdot \rangle)$. The action of $W$ on
$\mathbb R c$ is trivial. The affine plane $d+\dot{\mathfrak h}_{\mathbb R}$
is $W$-stable, and this action identifies $W$ with the affine Weyl
group $\dot W \propto \check{\dot T}$ and its affine action (\ref{affineaction}) on $\dot{\mathfrak h}_{\mathbb R}$
(see Chapter 5 of \cite{PS}). In this realization
\begin{equation}\label{r0eqn}r_{\alpha_0}= \dot h_{\dot{\theta}}\circ r_{\dot{\theta}}
,\quad \text{and}\quad r_{\alpha_i}=r_{\dot{\alpha_i}},\quad
i>0.\end{equation}

In general, given $w\in W$, we write
$$w=h_w\circ \dot w, \quad h\in\ \check{\dot T},\quad \dot w\in \dot W$$
and we let $Inv(w)$ denote the inversion set of $w$, i.e. the set
of positive roots which are mapped to negative roots by $w$.

\begin{definition}\label{minimal} A sequence of simple reflections
$r_1,r_2,..$ in $W$ is called reduced if $w_n=r_nr_{n-1}..r_1$ is
a reduced expression for each $n$.
\end{definition}

\begin{lemma}\label{rootlemma}Given a reduced sequence of simple
reflections $\{r_j\}$, corresponding to simple positive roots
$\gamma_j$,
\begin{enumerate}
\item[(a)] the inversion set
$$Inv(w_n)=\{\tau_j=w_{j-1}^{-1}\cdot\gamma_j=r_1..r_{j-1}\cdot\gamma_j,\quad j=1,..,n\}$$
\item[(b)] $w_{k}\tau_n>0$, $k<n$.
\end{enumerate}
\end{lemma}

A reduced sequence of simple reflections determines a
non-repeating sequence of adjacent alcoves
\begin{equation}\label{alcovewalk}C_0, w_1^{-1} C_0,..,
w_{n-1}^{-1}C_0=r_1..r_{n-1}C_0,..,\end{equation} where the
step from $w_{n-1}^{-1}C_0$ to $w_n^{-1}C_0$ is implemented by the
reflection $r_{\tau_n}=w_{n-1}^{-1}r_nw_{n-1}$ (in particular the
wall between $C_{n-1}$ and $C_n$ is fixed by $r_{\tau_n}$).
Conversely, given a sequence of adjacent alcoves $(C_j)$ which is
minimal in the sense that the minimal number of steps to go from
$C_0$ to $C_j$ is $j$, there is a corresponding reduced sequence
of reflections.

\begin{definition}A reduced sequence of simple reflections
$\{r_j\}$ is affine periodic if, in terms of the identification of
$W$ with the affine Weyl group, (1) there exists $l$ such that
$w_l \in \check{\dot T}$ and (2) $w_{s+l}= w_s \circ w_l$, for all
$s$. We will refer to $w_l^{-1}$ as the period ($l$ is the length
of the period).
\end{definition}

\begin{remarks}$\phantom{a}$
\begin{enumerate}
\item[(a)] The second condition is equivalent to
periodicity of the associated sequence of simple roots
$\{\gamma_j\}$, i.e. $\gamma_{s+l}=\gamma_s$.
\item[(b)] In terms of the associated walk through alcoves, affine
periodicity means that the walk from step $l+1$ onward is the
original walk translated by $w_l^{-1}$.
\end{enumerate}
\end{remarks}

We now recall Theorem 3.5 of \cite{PP} (this is what we will need
in Section \ref{topstratum} for root subgroup factorization of generic loops in $\G_0$).

\begin{theorem}\label{periodic1}$\phantom{a}$
\begin{enumerate}
\item[(a)] There exists an affine
periodic reduced sequence $\{r_j\}_{j=1}^\infty$ of simple reflections such
that, in the notation of Lemma \ref{rootlemma},
\begin{equation}\label{flips}\{\tau_j:1\le j<\infty\}=\{q d^* -\dot{\alpha
}:\dot{\alpha }>0,q=1,2,\ldots\},\end{equation} i.e. such that the
span of the corresponding root spaces is $\dot{\mathfrak
n}^{-}(z\mathbb C[z])$. The period can be chosen to be any point
in $C \cap \check{\dot T}$.

\item[(b)] Given a reduced sequence as in (a), and a reduced expression
for $\dot {w}_0=r_{-N}\cdots r_0$ (where $\dot w_0$ is the longest
element of $\dot W$), the sequence
$$r_{-N},\ldots,r_0,r_1,\ldots$$
is another reduced sequence. The corresponding set of positive
roots mapped to negative roots is
$$\{q d^* +\dot{\alpha }:\dot{\alpha }>0,q=0,1,\ldots\}
,$$ i.e. the span of the corresponding root spaces is $\dot
{\mathfrak n}^{+}(\mathbb C[z])$.
\end{enumerate}
\end{theorem}

\subsection{The Basic Framework and Notation}

In the remainder of the paper we will mainly be concerned with a slightly restricted loop analogue of
(\ref{fdgroupdiagram}):
\begin{equation}\label{loopgroupdiagram}
\xymatrix{
 & G & \\
G_0 \ar[ur] & & U \ar[ul] \\
& K \ar[ur] \ar[ul] &
}
\end{equation}
where $U:=\tilde{L}\dot U$, the (simply connected) central circle
extension of $L\dot U$, $G:=\tilde{L}\dot {G}$, the (simply connected) central $\mathbb C^*$ extension of $L\G$,
$G_0:=(\tilde{L}\G_0)_0$, the identity component of the central circle extension of $L\G_0$, and
$K:= (\tilde{L}\K)_0$, the identity component of the central circle extension of $L\K$. There is a corresponding
diagram of Lie algebras, where the Lie algebra of $G$ is $\mathfrak g=\tilde{L}\dot{\mathfrak g}$,
and so on.

It will often happen that we can more simply work at the level of loops, rather than at the level of central extensions. We will often state results, for example, in terms of $G$, but in proving results it is often possible and easier
to work with $L\G$.

\subsection{Contrast with Finite Dimensions}

In \cite{CP2}, in which we considered $\G_0$ (constant loops), the key fact (depending on the inner assumption) is that
\[
\dot{\mathfrak n}^{\pm}=\dot{\mathfrak n}^{\pm}_{\dot{\mathfrak k}}+ \dot{\mathfrak n}^{\pm}_{\dot{\mathfrak p}}
\]
where the latter summand, $\dot{\mathfrak n}^{\pm}_{\dot{\mathfrak p}}=\dot{\mathfrak n}^{\pm}\cap \dot{\mathfrak p}^{\mathbb C}$, is an abelian ideal in the parabolic subalgebra $\dot{\mathfrak k}^{\mathbb C}+ \dot{\mathfrak n}^{\pm}_{\dot{\mathfrak p}}$ of $\dot{\mathfrak g}$. This leads to a block (coarse) triangular factorization, which largely reduces the (finite dimensional) inner noncompact case to the compact case.

In the present context there is an analogous decomposition
\[
\mathfrak n^{\pm}=\mathfrak n^{\pm}_{\mathfrak k}+ \mathfrak n^{\pm}_{\mathfrak p}.
\]
In this case
\[
\mathfrak n^{+}_{\mathfrak p}=(\mathfrak n^{+}\cap L^+\dot{\mathfrak n}^-_{\dot{\mathfrak p}})+
(\mathfrak n^{+}\cap L^+\dot{\mathfrak n}^-_{\dot{\mathfrak p}})
\]
where each of the two summands is a subalgebra, but the sum
is not a Lie algebra, let alone an abelian ideal in a parabolic subalgebra. This is one reason why Birkhoff factorization and root subgroup factorization are so much more subtle for loops into $\dot G_0$ than for $\dot G_0$ itself.

\subsection{Compact vs Noncompact type roots in $\g$}

As before, a root of $\h$ on $\g$ is said to be of compact type if the corresponding root space belongs to $\k^\C$, and said to be of non-compact type if the corresponding root space belongs to $\p^\C$.  Here $\k^\C=\widetilde{L}\dot{\k}^\C$ and $\p^\C=\n_\p^-+\n_\p^+$ (so this terminology is perhaps less than ideal).

\begin{example}
In rank one, the compact type roots are the imaginary roots and the noncompact type roots are the real roots.  This is yet another special feature of the rank one case.
\end{example}

\section{Birkhoff Decomposition for Loops}\label{noncpt1}

By definition the Birkhoff decomposition of $ G=\widetilde{L}\G$ is
\begin{equation}\label{birkhoff}
G=\bigsqcup_{ W} \Sigma^{ G}_{
w} \text{ where }\Sigma^{ G}_{ w}:=N^-wB^+.
\end{equation}
If we fix a representative $\mathbf w\in N_{ U}( T)$ for $ w\in W$,
then each $g\in\Sigma^{G}_{w}$ has a unique Birkhoff factorization
\begin{equation}\label{birkhoff4}
g=l\mathbf w ma u,\quad l\in N^-\cap wN^-w^{-1},\quad ma\in
TA, u\in N^+.
\end{equation}
As in the finite dimensional case, for fixed $m_0\in T$, $\{g\in \Sigma^{ G}_{w}: m(g)=m_0\}$ is a stratum (diffeomorphic to the product of the Birkhoff stratum for the flag space $G/B^+$ corresponding to $w$ with
$N^+$); see Theorem 8.7.2 of \cite{PS}. We will refer to $\Sigma^{G}_{w}$ as the ``(isotypic) component
of the Birkhoff decomposition of $G$ corresponding to $w\in W$."

One virtue of root subgroup factorization is that it generates many explicit examples of Birkhoff factorizations.

\subsection{Birkhoff Decomposition for $L\dot U$}

Given $w\in W$, define
$$\Sigma^{U}_{w}:=\Sigma^{G}_{w}\cap U.$$

\begin{theorem}\label{compact2} Fix a representative $\mathbf w\in N_{U}(T)$ for $ w$.
For $g\in\Sigma^{ G}_{ w}$ the unique factorization (\ref{birkhoff4}) induces a bijective correspondence
\[
\Sigma^{U}_{w}\leftrightarrow \left( N^-\cap w N^-w^{-1}\right)\times T\text{ given by }g\mapsto (l,m).
\]
\end{theorem}

We refer to $\Sigma^{U}_{w}$ as the isotypic component of the Birkhoff decomposition
for $U$; each component consists of a union of strata permuted by the action of $T$. The theorem provides
an explicit parameterization for these strata. We have recalled this result simply for the sake of comparison. Our primary objective is to investigate the Birkhoff decomposition for $L\G_0$.

\subsection{Birkhoff Decomposition for $G_0:=(\tilde L\G_0)_0$, the Identity Component}

Given $w\in W$, define
$$\Sigma^{G_0}_{w}:=\Sigma^{G}_{w}\cap G_0$$
$$\Sigma^{L\G_0}_{w}:=\Sigma^{L\G}_{w}\cap L\G_0$$
and so on.

As we stated in the introduction (where we focused on the rank one case), our original expectation was that each of these components
would be (modulo a torus) contractible to $w$. Our main objective in this subsection is to provide examples
in the rank one case, for the identity component, which illustrate why this is not true.

\begin{proposition} $\Sigma^{LSU(1,1)_{(0)}}_{1}$ is properly contained in $LSU(1,1)_{(0)}$.
\end{proposition}

\begin{proof}  For any $g\in LSU(1,1)$ there is a pointwise polar decomposition
$$
g=\left(\begin{matrix}\lambda&0\\0&\lambda^{-1}\end{matrix}\right)
\left(\begin{matrix}a&b\\ b^*& a\end{matrix}\right)
$$
where $a=\sqrt{1+\vert b\vert^2}$, and $\lambda:S^1\to S^1$.

If $g\in LSU(1,1)_{(0)}$, then $\lambda$ has degree zero, and thus $\lambda$ has a triangular factorization
\[
\lambda=e^{\psi_-}e^{\psi_0}e^{\psi_+}
\]
where $\psi_-=-\psi_+^*$ and $\psi_0\in i\mathbb R$. Because $a$ is a positive periodic function,
it will have a triangular factorization
\[
a=e^{\chi_-}e^{\chi_0}e^{\chi_+}
\]
where $\chi_-=\chi_+^*$ and $\chi_0\in \mathbb R$.

We can always multiply $g$ on the left (right) by something in $B^-$ ($B^+$, respectively) without affecting the question of whether $g$ has a triangular factorization. For example in determining whether $g$ has a triangular factorization, we can ignore the factor $\exp(\psi_-+\psi_0)$ in $\lambda$, because this can be factored out on the left. We will use this observation repeatedly (note that we can recover $\psi_-$ from $\psi_+$, and the zero mode is inconsequential).

There is a factorization of
\[
\left(\begin{matrix}a&b\\\bar b& a\end{matrix}\right)
\]
as the product
\[
\left(\begin{matrix}e^{\chi_-}&0\\0&e^{-\chi_-}\end{matrix}\right)
\left(\begin{matrix}1&0\\ b^*e^{(\chi_--\chi_0-\chi_+)} & 1\end{matrix}\right)\left(\begin{matrix}e^{\chi_0}&0\\ 0&e^{-\chi_0}\end{matrix}\right)
\left(\begin{matrix}1&b e^{(-\chi_--\chi_0+\chi_+)} \\ 0& 1\end{matrix}\right)
\left(\begin{matrix}e^{\chi_+}&0\\0&e^{-\chi_+}\end{matrix}\right).
\]
To obtain $g$ we have to multiply this on the left by $\lambda$. It follows after some calculation that $g$ will have a triangular factorization if and only if
\[
\left(\begin{matrix}1&0\\ b^*e^{(\chi_--\chi_0-\chi_+)}e^{-2\psi_+}& 1\end{matrix}\right)\left(\begin{matrix}e^{\chi_0}&0\\ 0&e^{-\chi_0}\end{matrix}\right)\left(\begin{matrix}1&b e^{(-\chi_--\chi_0+\chi_+)}e^{2\psi_+}\\ 0& 1\end{matrix}\right)
\]
has a triangular factorization.

At this point, to simplify notation, we let $b_1:=b e^{(-\chi_--\chi_0+\chi_+)}$. Note that $b_1b_1^*= bb^*e^{(-2\chi_0)}$. Thus $g$ has a triangular factorization if and only if the loop
\[
\left(\begin{matrix}1&0\\ b_1^*e^{-2\psi_+}& 1\end{matrix}\right)\left(\begin{matrix}e^{\chi_0}&0\\ 0&e^{-\chi_0}\end{matrix}\right)\left(\begin{matrix}1&b_1 e^{2\psi_+}\\ 0& 1\end{matrix}\right)=e^{\chi_0}\left(\begin{matrix}1&b_1e^{2\psi_+}\\ b_1^*e^{-2\psi_+}& b_1b_1^*+e^{-2\chi_0}\end{matrix}\right)
\] has a triangular factorization. Note that the $(2,2)$ entry of the right hand side equals $aa^*e^{-2\chi_0}$.

We directly calculate the kernel of the Toeplitz operator associated to this loop. We obtain the equations (for $f_1,f_2\in H^0(D)$)
\[
f_1+( b_1e^{2\psi_+}f_2)_+=0,\text{ and } \left(b_1^*e^{-2\psi_+}f_1+(b_1b_1^*+e^{-2\chi_0})f_2\right)_+=0.
\]
We can solve the first equation for $f_1$. The second equation becomes
\[
\left((e^{-2\chi_0}+b_1b_1^*)f_2-b_1^* e^{-2\psi_+}(b_1e^{2\psi_+}f_2)_+\right)_+=0.
\]
If we set $b_2=b_1e^{\chi_0}=be^{-\chi_-+\chi_+}$, then this can be rewritten as
\[
\left(f_2+b_2^* e^{-2\psi_+}(b_2e^{2\psi_+}f_2)_-\right)_+=0.
\]
If we set $F=e^{2\psi_+}f_2$, then we see that
there exists a nontrivial kernel if and only if there exists nonzero $F\in H_+$ such that
\begin{equation}\label{tricondition}
\left(e^{-2\psi_+}(F+b_2^*(b_2F)_-)\right)_+=0.
\end{equation}
It is easy to find $\psi_+$ and $b_2$ such that there does exist a nonzero $F$ satisfying this condition.

\begin{example} $b_2=\frac{1}{z}-1$, $e^{-2\psi_+}=e^{2z}$, and $F=\frac{1-e^{-2z}}{z}-1$.
In other words if
\[
g=\left(\begin{matrix}e^{\frac1{z}-z}&0\\0&e^{-\frac{1}{z}+z}\end{matrix}\right)
\left(\begin{matrix}(3-\frac1{z}-z)^{1/2}&e^{-\chi_-+\chi_+}(\frac{1}{z}-1)\\e^{+\chi_--\chi_+}(z-1)&(3-\frac1{z}-z)^{1/2}
\end{matrix}\right)
\] where
\[
(3-\frac1{z}-z)^{1/2}=e^{\chi_-}e^{\chi_0}e^{\chi_+}
\]
then $g$ is a loop in the identity component of $LSU(1,1)$ and does not have a Riemann-Hilbert factorizaton, hence also
does not have a triangular factorization.
\end{example}
\end{proof}

\subsection{Birkhoff Decomposition for nonidentity components of $L\G_0$}

Consider the rank one case and the problem of finding the Birkhoff factorization for $g$ which is of the form $g=\left(\begin{matrix}z^{-n}&0\\0&z^n\end{matrix}\right)g_0$, where $g_0$ is in the identity component and has a known triangular factorization (as for example in Theorem \ref{SU(1,1)theorem1}), and $n>0$. Write
$$g=\left(\begin{matrix}z^{-n}&0\\0&z^n\end{matrix}\right)
\left(\begin{matrix}l_{11}&l_{12}\\l_{21}&l_{22}\end{matrix}\right)\left(\begin{matrix}m_0a_0&0\\0&(m_0a_0)^{-1}\end{matrix}\right)
\left(\begin{matrix}u_{11}&u_{12}\\u_{21}&u_{22}\end{matrix}\right).$$
Factor $l$ as
$$\left(\begin{matrix}l_{11}&l_{12}\\l_{21}&l_{22}\end{matrix}\right)=
\left(\begin{matrix}L_{11}&L_{12}\\\sum_{k=-\infty}^{-2n}\alpha_kz^k&L_{22}\end{matrix}\right)
\left(\begin{matrix}1&0\\\sum_{k=-2n+1}^{-1}x_kz^k&1\end{matrix}\right).$$
Then $g$ will have the form
$$g=L'\left(\begin{matrix}z^{-n}&0\\0&z^n\end{matrix}\right)\left(\begin{matrix}1&0\\\sum_{-2n+1}^{0}x_kz^k&1\end{matrix}\right)
\left(\begin{matrix}m_0a_0&0\\0&(m_0a_0)^{-1}\end{matrix}\right)
\left(\begin{matrix}u_{11}&u_{12}\\u_{21}&u_{22}\end{matrix}\right)$$
where $L'\in N^-$. Consequently to find the Birkhoff factorization for $g$, it suffices
to find the factorization for the triangular matrix valued function
\begin{equation}\label{triangularloop}\left(\begin{matrix}z^{-n}&0\\0&z^n\end{matrix}\right)
\left(\begin{matrix}1&0\\\sum_{k=-2n+1}^{0}x_kz^k&1\end{matrix}\right)=
\left(\begin{matrix}z^{-n}&0\\\sum_{-n+1}^{n}c_{k}z^k&z^n\end{matrix}\right).
\end{equation}

\begin{remark}What we are doing here is factoring $N^-$ as $N^-\cap wN^-w^{-1}$ times $N\cap wN^+w^{-1}$. So this is very general. The problem of understanding Birkhoff factorization for
triangular matrix valued functions is considered in \cite{CG}.
\end{remark}

\begin{example} When $n=1$, we could take $g_0=g_1$ in Theorem \ref{SU(1,1)theorem1}. Then
$$g=\left(\begin{matrix}z^{-1}&0\\0&z\end{matrix}\right)\mathbf a(\eta_1)\begin{pmatrix} 1&\bar{\eta}_1\\
\eta_1&1\end{pmatrix}\mathbf
a(\eta_0)\begin{pmatrix} 1&
\bar{\eta}_0z\\
\eta_0z^{-1}&1\end{pmatrix},
$$
$$=\left(\begin{matrix}z^{-1}&0\\0&z\end{matrix}\right)
\begin{pmatrix} 1&0\\
\bar y_0+ \bar y_1z^{-1}&1\end{pmatrix}\begin{pmatrix} a_1&0\\
0&a_1^{-1}\end{pmatrix}\begin{pmatrix} \alpha_1 (z)&\beta_1 (z)\\
\gamma_1 (z)&\delta_1 (z)\end{pmatrix}$$
where $y_1=-\bar \eta_1$ and $y_0=-\bar \eta_0(1-\eta_1\bar \eta_1)$ (note $\vert y_0\vert, \vert y_1\vert<1$).
\end{example}

\begin{lemma}\label{roughclaim} Fix $n>0$. For a triangular matrix valued function as in (\ref{triangularloop}),
\begin{enumerate}
\item[(a)] the Toeplitz operator $A$ is invertible if and only if the Toeplitz matrix
\begin{equation}\label{T_matrix}
A'=\left(\begin{matrix}c_0&c_{-1}&..&c_{-n+1}\\c_{1}&c_0&..&c_{-n+2}\\
..&..&..&..\\c_{n-1}&..&..&c_0\end{matrix}\right)
\end{equation}
is invertible, and
\item[(b)] the shifted Toeplitz operator $A_1$ is invertible if and only if the Toeplitz matrix
\begin{equation}\label{shifted_T_matrix}
A''=\left(\begin{matrix}c_{1}&c_{0}&..&c_{-n+2}\\c_{2}&c_{1}&..&c_{-n+3}\\
..&..&..&..\\c_{n}&..&..&c_{1}\end{matrix}\right)
\end{equation}
is invertible.
\end{enumerate}
\end{lemma}

\begin{proof} The Fredholm indices for both operators are zero, so we need to check the kernels.

Part (a): Suppose that $$\left(\begin{matrix}f\\h\end{matrix}\right)=\left(\begin{matrix}\sum_{k=0}^{\infty}f_kz^k\\\sum_{k=0}^{\infty}h_kz^k\end{matrix}\right)$$
is in the kernel of $A$. Then $(z^{-n}f)_+=0$, implying $f=\sum_{k=0}^{n-1}f_kz^k$, and
$$\left((\sum_{k=-n+1}^{n}c_kz^k)(\sum_{k=0}^{n-1}f_kz^k)+\sum_{k=0}^{\infty}h_kz^{k+n}\right)_+=0.$$
This equation implies $h_k=0$ for $k\ge n$.  These equations have the matrix form

$$\left(\begin{matrix}A'&0\\C'&1_{n\times n}\end{matrix}\right)\left(\begin{matrix}\vec{f}\\\vec{h}\end{matrix}\right)=\vec{0}$$
where $\vec{f}$ (resp. $\vec{h}$) is the vector of coefficients of $f$ (resp. $h$) and $A'$ is the $n\times n$ Toeplitz matrix in (\ref{T_matrix}).
This implies part (a).

Part (b): Suppose that $$\left(\begin{matrix}f\\h\end{matrix}\right)=\left(\begin{matrix}\sum_{k=0}^{\infty}f_kz^k\\\sum_{k=1}^{\infty}h_kz^k\end{matrix}\right)$$
is in the kernel of $A_1$. Then $(z^{-n}f)_+=0$, implying $f=\sum_{k=0}^{n-1}f_kz^k$, and
$$\left((\sum_{k=-n+1}^{n}c_kz^k)(\sum_{k=0}^{n-1}f_kz^k)+\sum_{k=1}^{\infty}h_kz^{k+n}\right)_+=0.$$
These equations have the matrix form

$$\left(\begin{matrix}A''&0\\C''&1_{n\times n}\end{matrix}\right)\left(\begin{matrix}\vec{f}\\\vec{h}\end{matrix}\right)=\vec{0}$$
where $A''$ is the $n\times n$ Toeplitz matrix in (\ref{shifted_T_matrix}).
$$A''=\left(\begin{matrix}c_{1}&c_{0}&..&c_{-n+2}\\c_{2}&c_{1}&..&c_{-n+3}\\
..&..&..&..\\c_{n}&..&..&c_{1}\end{matrix}\right).$$
This implies part (b).
\end{proof}

\begin{example} Suppose $n=1$.  When $c_0\ne 0$ there is a Riemann-Hilbert factorization (because $A$ is invertible)
$$ \left(\begin{matrix}z^{-1}&0\\ c_0+c_1z&z\end{matrix}\right)= \left(\begin{matrix}1&z^{-1}/c_0\\0&1\end{matrix}\right)
 \left(\begin{matrix}-c_1/c_0&-1/c_0\\c_0&0\end{matrix}\right) \left(\begin{matrix}1+\frac{c_1}{c_0}z&z/c_0\\\frac{-c_1^2}{c_0}z&1-\frac{c_1}{c_0}z\end{matrix}\right).   $$
When $c_0,c_1\ne 0$, there is a triangular factorization (because $A$ and $A_1$ are invertible),
$$ \left(\begin{matrix}z^{-1}&0\\ c_0+c_1z&z\end{matrix}\right)= \left(\begin{matrix}1-\frac{c_0}{c_1}z^{-1}&\frac1{c_0}z^{-1}\\-\frac{c_0^2}{c_1}&1\end{matrix}\right)
 \left(\begin{matrix}-\frac{c_1}{c_0}&0\\0&-\frac{c_0}{c_1}\end{matrix}\right)
 \left(\begin{matrix}1&\frac1{c_0}\\-\frac{c_1^2}{c_0}z&1-\frac{c_1}{c_0}z\end{matrix}\right).   $$
In this case $g\in \Sigma^{LSU(1,1)_{(-1)}}_1$.

When $c_1\to 0$ this ``degenerates" to a Birkhoff factorization
$$\left(\begin{matrix}z^{-1}&0\\c_0&z\end{matrix}\right)= \left(\begin{matrix}1&\frac1{c_0}z^{-1}\\0&1\end{matrix}\right)
 \left(\begin{matrix}0&-1\\1&0\end{matrix}\right) \left(\begin{matrix}\frac1{c_0}&0\\0&c_0\end{matrix}\right)\left(\begin{matrix}1&\frac1{c_0}z\\0&1\end{matrix}\right).   $$ In this case $g\in \Sigma^{LSU(1,1)_{(-1)}}_{r_1}$.

When $c_0\to 0$ this ``degenerates" to a Birkhoff factorization
$$\left(\begin{matrix}z^{-1}&0\\c_1z&z\end{matrix}\right)= \left(\begin{matrix}1&\frac1{c_1}z^{-2}\\0&1\end{matrix}\right)
 \left(\begin{matrix}0&-z^{-1}\\z&0\end{matrix}\right) \left(\begin{matrix}\frac1{c_1}&0\\0&c_1\end{matrix}\right)\left(\begin{matrix}1&\frac1{c_1}\\0&1\end{matrix}\right).   $$
In this case $g\in \Sigma^{LSU(1,1)_{(-1)}}_{r_0}$.

When both $c_0,c_1\to 0$ this goes to $\left(\begin{matrix}z^{-1}&0\\0&z\end{matrix}\right)$.In this case $g\in \Sigma^{LSU(1,1)_{(-1)}}_{r_0r_1}$, where in the Weyl group $\left(\begin{matrix}z^{-1}&0\\0&z\end{matrix}\right) =r_0r_1  $.
\end{example}

These calculations
show that we are obtaining loops in the corresponding strata, despite the fact that neither $r_0$ nor
$r_1$ are represented by loops in $\K=S^1$. Moreover the conditions on $c_0,c_1$ above show that there is something
topologically nontrivial about the intersection of the $\Sigma_{r_1}$ component with the $n=-1$ connected component.

\section{Root subgroup factorization for generic loops in $\G_0$}\label{topstratum}

Our objective in this section is to prove analogues of Theorems 4.1, 4.2, and 5.1 of \cite{PP}, for
generic loops in $\G_0$ (which is always assumed to be of inner type). The structure of the proofs
in this noncompact context is basically the same as in \cite{PP}. But there are important differences.
In order to obtain formulas for determinants of Toeplitz operators, as in Theorem \ref{Toeplitz},
we have to work with the central extension $\widetilde L\dot G$.

Throughout this section we choose a reduced
sequence $\{r_j\}_{j=1}^\infty$ as in Theorem \ref{periodic1}, part (a). We set $\mathbf w_j=\mathbf r_j...\mathbf r_1$ and
$$i_{\tau_n}=\mathbf w_{n-1} i_{\gamma_n} \mathbf
w_{n-1}^{-1}, \qquad n=1,2,\ldots$$
$$i_{\tau_{-N}^{\prime}}=i_{\gamma_{-N}},\quad
i_{\tau_{-(N-1)}^{\prime}}=\mathbf
r_{-N}i_{\gamma_{-(N-1)}}\mathbf r_{-N}^{-1},\ldots,\quad
i_{\tau_0^{\prime}}=\dot{\mathbf w}_0 i_{\gamma_0} \dot{\mathbf
w}_0^{-1}
$$ and for $n>0$
$$i_{\tau_n^{\prime}}=\dot{\mathbf w}_0 \mathbf w_{n-1}
i_{\gamma_n} \mathbf w_{n-1}^{-1} \dot{\mathbf w}_0^{-1}.$$

As in \cite{CP2}, for $\zeta\in \mathbb
C$, let $\mathbf a_+(\zeta)=(1+\vert\zeta\vert^2)^{-1/2}$  and
\begin{equation}\label{kfactor}
k(\zeta)=\mathbf a_+(\zeta)
\begin{pmatrix}1&-\bar{\zeta}\\\zeta&1
\end{pmatrix}=\begin{pmatrix}1&0\\\zeta&1
\end{pmatrix}\begin{pmatrix}\mathbf a_+(\zeta)&0\\0&\mathbf a_+(\zeta)^{-1}
\end{pmatrix}\begin{pmatrix}1&-\bar{\zeta}\\0&1
\end{pmatrix} \in SU(2).
\end{equation}
For $|\zeta|<1$, let $\mathbf a_-(\zeta)=(1-\vert\zeta\vert^2)^{-1/2}$ and
\begin{equation}\label{qfactor}
q(\zeta)=\mathbf a_-(\zeta)
\begin{pmatrix}1&\bar{\zeta}\\\zeta&1
\end{pmatrix}=\begin{pmatrix}1&0\\\zeta&1
\end{pmatrix}\begin{pmatrix}\mathbf a_-(\zeta)&0\\0&\mathbf a_-(\zeta)^{-1}
\end{pmatrix}\begin{pmatrix}1&\bar{\zeta}\\0&1
\end{pmatrix} \in SU(1,1).
\end{equation}

\subsection{Generalizations of Theorem \ref{SU(1,1)theorem1}}\label{topstratum1}

\begin{theorem}\label{Gtheorem1} Suppose that $\widetilde{g}_1 \in \widetilde{L}_{fin}\dot{G}_0$
and $\Pi(\widetilde{g}_1)=g_1$. Consider the following three statements:
\begin{enumerate}
\item[(I.1)] $m(\widetilde{g}_1)=1$, and for each complex irreducible
representation $V(\pi)$ for $\dot{G}$, with lowest weight vector
$\phi \in V(\pi)$, $\pi(g_1)^{-1}(\phi)$ is a polynomial in $z$
(with values in $V$), and is a positive multiple of $\phi$ at
$z=0$.
\item[(I.2)] $\widetilde{g}_1$ has a factorization of the form
\[
\widetilde{g}_1=i_{\tau'_n}(g(\eta_n))..i_{\tau'_{-N}}(g(\eta_{-N}))\in \widetilde{L}_{fin}\dot G_0
\]
where $g(\eta_j)=k(\eta_j)$ for some $\eta_j\in\C$ (resp. $g(\eta_j)=q(\eta_j)$ for some $\eta_j\in\Delta$) when $\tau_j$ is a compact type (resp. non-compact type) root.
\item[(I.3)] $\widetilde{g}_1$ has triangular factorization of the form
$\widetilde g_1=l_1a_1u_1$ where $l_1\in \dot N^-(\mathbb
C[z^{-1}])$.
\end{enumerate}
Then statements (I.1) and (I.3) are equivalent. (I.2) implies (I.1) and (I.3).

Moreover, in the notation of (I.2),
$$a_1=\prod_{j=-N}^n \mathbf a(\eta_j)^{h_{\tau^{\prime}_j}}.$$

Similarly, suppose that $\widetilde{g}_2 \in \widetilde{L}_{fin}\dot{G}_0$
and $\Pi(\widetilde{g}_2)=g_2$. Consider the following three statements:

\begin{enumerate}
\item[(II.1)] $m(\widetilde g_2)=1$, and for each complex irreducible
representation $V(\pi)$ for $\dot{G}$, with highest weight vector
$v \in V(\pi)$, $\pi(g_2)^{-1}(v)$ is a polynomial in $z$ (with
values in $V$), and is a positive multiple of $v$ at $z=0$.

\item[(II.2)] $\widetilde g_2$ has a factorization of the form
$$\widetilde g_2=i_{\tau_n}(g(\zeta_n))..i_{\tau_1}(g(\zeta_1))$$
for some $\zeta_j \in \Delta$.

\item[(II.3)] $\widetilde g_2$ has triangular factorization of the form
$\widetilde g_2=l_2a_2u_2$, where $l_2 \in \dot N^+(z^{-1}\mathbb
C [z^{-1}])$.
\end{enumerate}
Then statements (II.1) and (II.3) are equivalent. (II.2) implies (II.1) and (II.3).

Also, in the notation of (II.2),
\begin{equation}\label{productformula}
a_2=\prod_{j=1}^n \mathbf a(\zeta_j)^{h_{\tau_j}}.
\end{equation}
\end{theorem}

\begin{remark} Note that we are not making any attempt to characterize the set of
$l_1$ that arise in (I.3) (and similarly for the set of $l_2$ in (II.3)).
\end{remark}

\begin{conjecture}\label{conjecture1}
If $g_1$ is in the
identity connected component of the sets in (I.1) and (I.3), then the converse holds, i.e. $g_1$ has a root subgroup
factorization as in (I.2).  If $g_2$ is in the
identity connected component of the sets in (II.1) and (II.3), then the converse holds, i.e., $g_2$ has a root subgroup
factorization as in (II.2).
\end{conjecture}

In the course of the following proof of Theorem \ref{Gtheorem1}, we will prove a version of this conjecture, in the rank one case, which
completes the proof of Theorem \ref{SU(1,1)theorem1} (see Remark \ref{proofofconjecture1} below).

\begin{proof}  The two sets of implications are proven in the same way.
We consider the second set.

We first want to argue that (II.2) implies (II.3). We recall that the
subalgebra $\mathfrak n^- \cap \mathbf w_{n-1}^{-1} \mathfrak
n^+ \mathbf w_{n-1}$ is spanned by the root spaces corresponding
to negative roots $-\tau_j$, $j=1,..,n$. The calculation is the
same as in the proof of Theorem 2.5 in \cite{CP2}. In the process we will also prove the
product formula for $a_2$.

The equation (\ref{kfactor}) implies that
\begin{eqnarray*}
\iota_{\tau_j}(g(\zeta_j)) & = & \iota_{\tau_j}(
\begin{pmatrix}1&0\\\zeta_j&1
\end{pmatrix})\mathbf a(\zeta_j)^{h_{\tau_j}}\iota_{\tau_j}(
\begin{pmatrix}1&\pm\bar{\zeta_j}\\0&1
\end{pmatrix}) \\
& = & \exp(\zeta_j
f_{\tau_j})\mathbf a(\zeta_j)^{h_{\tau_j}}\mathbf
w_{j-1}^{-1}\exp(\pm\bar{\zeta}_j e_{\gamma_j})\mathbf w_{j-1}
\end{eqnarray*}
is a triangular factorization.  Here, $\mathbf a(\zeta_j)=\mathbf a_{\pm}(\zeta_j)$ and the plus/minus case is used when $\tau_j$ is a compact/noncompact type root, respectively.

Let $g^{(n)}=\iota_{\tau_n}(g(\zeta_n))..\iota_{\tau_1}(g(\zeta_1))$.
First suppose that $n=2$. Then
\begin{equation}\label{loopn=2case}
g^{(2)}=
\exp(\zeta_2 f_{\tau_2})\mathbf a(\zeta_2)^{h_{\tau_2}}\mathbf
r_1\exp(\pm\bar{\zeta}_2 e_{\gamma_2})\mathbf r_1^{-1} \exp(\zeta_1
f_{\gamma_1})\mathbf a(\zeta_1)^{h_{\gamma_1}}\exp(\pm\bar{\zeta}_1
e_{\gamma_1}).
\end{equation}
The key point is that
\begin{eqnarray*}
\mathbf r_1\exp(\pm\bar{\zeta}_2 e_{\gamma_2})\mathbf r_1^{-1} \exp(\zeta_1
f_{\gamma_1}) & = & \mathbf r_1 \exp(\pm\bar{\zeta}_2 e_{\gamma_2})
\exp(\zeta_1 e_{\gamma_1})\mathbf r_1^{-1} \\
& = & \mathbf r_1 \exp(\zeta_1
e_{\gamma_1})\widetilde u \mathbf r_1^{-1},\quad (\text{for some}
\quad \widetilde u\in N^+\cap r_1 N^+ r_1^{-1}) \\
& = & \exp(\zeta_1
f_{\gamma_1})\mathbf u, \quad (\text{for some} \quad \mathbf u\in
N^+).
\end{eqnarray*}
Insert this calculation into (\ref{loopn=2case}). We then see
that $g^{(2)}$ has a triangular factorization $g^{(2)}=l^{(2)}a^{(2)}u^{(2)}$, where
\[
a^{(2)}=
\mathbf a(\zeta_1)^{h_{\tau_1}}\mathbf a(\zeta_2)^{h_{\tau_2}}
\]
and
\begin{equation}\label{2ndcase}
l^{(2)}=\exp(\zeta_2
f_{\tau_2})\exp(\zeta_1\mathbf
a(\zeta_2)^{-\tau_1(h_{\tau_2})}f_{\tau_1})\end{equation}
$$=\exp(\zeta_2
f_{\tau_2}+\zeta_1\mathbf
a(\zeta_2)^{-\tau_1(h_{\tau_2})}f_{\tau_1})$$ (the last equality
holds because a two dimensional nilpotent algebra is necessarily
commutative).

To apply induction, we assume that $g^{(n-1)}$ has a triangular
factorization $g^{(n-1)}=l^{(n-1)}a^{(n-1)}u^{(n-1)}$ with
\begin{equation}\label{induction}l^{(n-1)}=\exp(\zeta_{n-1}
f_{\tau_{n-1}})\widetilde l \in N^-\cap
w_{n-1}^{-1}N^+w_{n-1}=\exp(\sum_{j=1}^{n-1}\mathbb C
f_{\tau_j}),\end{equation} for some $\widetilde l \in N^-\cap
w_{n-2}^{-1}N^+w_{n-2}=\exp(\sum_{j=1}^{n-2}\mathbb C f_{\tau_j})$,
and
\[
a^{(n-1)}= \prod_{j=1}^{n-1}\mathbf a(\zeta_j)^{h_{\tau_j}}.
\]
We have established this for $n-1=1,2$. For $n \ge 2$
\begin{eqnarray*}
g^{(n)} & = &\exp(\zeta_n f_{\tau_n})\mathbf a(\zeta_n)^{h_{\tau_n}}\mathbf w_{n-1}^{-1}
\exp(\pm\bar{\zeta}_n e_{\gamma_n})\mathbf w_{n-1}\exp(\zeta_{n-1}
f_{\tau_{n-1}})\widetilde l a(g^{(n-1)})u(g^{(n-1)})
\\
& = &\exp(\zeta_n f_{\tau_n})\mathbf a(\zeta_n)^{h_{\tau_n}}\mathbf w_{n-1}^{-1}
\exp(\pm\bar{\zeta}_ne_{\gamma_n}) \widetilde u\mathbf
w_{n-1}a(g^{(n-1)})u(g^{(n-1)}),
\end{eqnarray*}
where $\widetilde u= \mathbf
w_{n-1}\exp(\zeta_{n-1} f_{\tau_{n-1}})\widetilde l \mathbf
w_{n-1}^{-1}\in \mathbf w_{n-1} N^{-} \mathbf w_{n-1}^{-1} \cap
N^+$. Now write
\[
\exp(\pm\bar{\zeta}_ne_{\gamma_n}) \widetilde u=\widetilde u_1\widetilde u_2,
\] relative to the decomposition
\[
N^+=\left(N^+\cap w_{n-1}N^-w_{n-1}^{-1}\right)\left(N^+\cap
w_{n-1}N^+w_{n-1}^{-1}\right).
\]
Let
\[
\mathbf l=\mathbf a(\zeta_n)^{h_{\tau_n}}\mathbf
w_{n-1}^{-1}\widetilde u_1 \mathbf w_{n-1}\mathbf
a(\zeta_n)^{-h_{\tau_n}}\in N^- \cap \mathbf w_{n-1}^{-1} N^+
\mathbf w_{n-1}.
\]
Then $g^{(n)}$ has triangular decomposition
\[
g^{(n)}=\left(\exp(\zeta_n f_{\tau_n})\mathbf
l\right) \left(\mathbf
a(\zeta_n)^{h_{\tau_n}}a^{(n-1)}\right)\left(
(a^{(n-1)})^{-1}\widetilde u_2a^{(n-1)}u^{(n-1)}\right).
\]
This implies the induction step.

This calculation shows that (II.2) implies (II.3). It also implies the product formula
for (\ref{productformula}) $a_2$.

\begin{remark}\label{proofofconjecture1} In reference to Conjecture \ref{conjecture1}, we observe that the preceding calculation shows that
we have a map (using the notation we have established above)
\begin{equation}\label{lowerstuff}
\{(\zeta_j):j=1,..,n\}\to \exp(\oplus_{j=1}^n\mathbb C f_{\tau_j}):(\zeta_j)\to l(g^{(n)})
\end{equation}
where $\zeta_j$ ranges over either the complex plane or a disk, depending on whether the $j$th root is of compact or noncompact type. The calculation also show that the map is 1-1 and open. We claim that the image
of this map
is closed in
\[
\{l_2\in \exp(\oplus_{j=1}^n\mathbb C f_{\tau_j}): \exists \quad\widetilde g_2 \text{ having triangular factorization } \widetilde g_2=l_2a_2u_2\}.
\]
This follows from the product formula for $a_2$, which shows that as the parameters tend to the boundary,
the triangular factorization fails. This implies that the image of the map is the connected component which
contains $l_2=1$. This does prove the implication (II.2) implies (II.3) in Theorem \ref{SU(1,1)theorem1}, because $n$ is fixed in
the statement of that theorem, but this does not
complete the proof of Conjecture \ref{conjecture1}. The difficulty is that we do not know how to formulate statements (I.1) and (II.1)
in the general case in a way that regards $n$ as fixed.
\end{remark}

It is obvious that (II.3) implies (II.1). In fact (II.3) implies a
stronger condition. If (II.3) holds, then given a highest weight
vector $v$ as in (II.1), corresponding to highest weight
$\dot{\Lambda}$, then
\begin{equation}\label{highestweight}
\pi(g_2^{-1})v=\pi(u_2^{-1}a_2^{-1}l^{-1})v=a_2^{-\dot{\Lambda}}\pi(u_2^{-1})v,\end{equation}
implying that $\pi(g_2^{-1})v$ is holomorphic in $\Delta$ and
nonvanishing at all points. However we do not need to include this
nonvanishing condition in (II.1), in this finite case.

It remains to prove that (II.1) implies (II.3). Because
$\widetilde g_2$ is determined by $g_2$, as in Lemma
\ref{notationlemma}, it suffices to show that $g_2$ has a
triangular factorization (with trivial $\dot T$ component). Hence
we will slightly abuse notation and work at the level of loops in
the remainder of this proof.

To motivate the argument, suppose that $g_2$ has triangular
factorization as in (II.3). Because $u_2(0)\in \dot{N}^+$, there
exists a pointwise $\dot{G}$-triangular factorization
\begin{equation}\label{pointwise}
u_2(z)^{-1}=
\dot{l}(u_2(z)^{-1})\dot{d}(u_2(z)^{-1})\dot{u}(u_2(z)^{-1})
\end{equation}
which is certainly valid in a neighborhood of $z=0$; more
precisely, (\ref{pointwise}) exists at a point $z\in \mathbb C$ if
and only if
\[
\dot{\sigma}_i(u_2(z)^{-1})\ne 0, \quad i=1,..,r.
\]
When (\ref{pointwise}) exists (and using the fact that $g_2$ is
defined on $\mathbb C^*$ in this algebraic context),
\[
g_2(z)=\left (l_2(z)a_2 \dot{u}(u_2(z)^{-1})^{-1} a_2^{-1} \right)
\left ( a_2 \dot{d}(u_2(z)^{-1})^{-1} \right )
\dot{l}(u_2(z)^{-1})^{-1}.
\]
This implies
\begin{equation}\label{k2pointwise}
g_2(z)^{-1}=\dot{l}(u_2(z)^{-1})\left (
\dot{d}(u_2(z)^{-1})a_2^{-1} \right ) \left (a_2
\dot{u}(u_2(z)^{-1}) a_2^{-1} l_2(z)^{-1}\right).
\end{equation}
This is a pointwise $\dot{G}$-triangular factorization of
$g_2^{-1}$, which is certainly valid in a punctured neighborhood
of $z=0$. The important facts are that (1) the first factor in
(\ref{k2pointwise})
\begin{equation}\label{removable}\dot{l}(g_2^{-1})=\dot{l}(u_2(z)^{-1})\end{equation}
does not have a pole at $z=0$; (2) for the third (upper
triangular) factor in (\ref{k2pointwise}), the factorization
\begin{equation}\label{unobstructed}\dot{u}(g_2^{-1}) ^{-1}= l_2(z) \left
(a_2\dot{u}(u_2(z)^{-1})a_2^{-1} \right )\end{equation} is a
$L\dot{G}$-triangular factorization of $\dot{u}(g_2^{-1})^{-1}\in
L\dot{N}^+$, where we view $\dot{u}(g_2^{-1})^{-1}$ as a loop by
restricting to a small circle surrounding $z=0$; and (3) because
there is an a priori formula for $a_2$ in terms of $g_2$ (see
(\ref{loopdiagonal})), we can recover $l_2$ and (the pointwise
triangular factorization for) $u_2^{-1}$ from
(\ref{k2pointwise})-(\ref{unobstructed}): $l_2=l(\dot
u(g_2^{-1})^{-1})$ (by (\ref{unobstructed})), and
\begin{equation}\label{backwards}\dot l(u_2(z)^{-1})=\dot l( g_2(z)^{-1}),
\quad \dot d(u_2(z)^{-1})=\dot d(g_2(z)^{-1})a_2,
\end{equation}
$$ \text{and} \quad \dot u(u_2(z)^{-1})=a_2^{-1}u(\dot
u(g_2(z)^{-1}))a_2.$$ We remark that this uses the fact that $g_2$
is defined in $\mathbb C^*$ in an essential way.

Now suppose that (II.1) holds. In particular (II.1) implies that
$\dot{\sigma}_i(g_2^{-1})$ has a removable singularity at $z=0$
and is positive at $z=0$, for $i=1,..,r$. Thus $g_2^{-1}$ has a
pointwise $\dot{G}$-triangular factorization as in
(\ref{k2pointwise}), for all $z$ in some punctured neighborhood of
$z=0$.

We claim that (\ref{removable}) does not have at pole at $z=0$. To
see this, recall that for an $n\times n$ matrix $g=(g_{ij})$
having an LDU factorization, the entries of the factors can be
written explicitly as ratios of determinants:
$$
\dot{d}(g)=\mathrm{diag}(\sigma_1,\sigma_2/\sigma_1,\sigma_3/\sigma_2,..,\sigma_
n/\sigma_{n-1})
$$
where $\sigma_k$ is the determinant of the
$k^{th}$ principal submatrix, $\sigma_k=\det((g_{ij})_{1\le i,j\le
k})$; for $i>j$,
\begin{equation}\label{lower}l_{ij}=\det\begin{pmatrix} g_{11}&g_{12}&..&g_{1j}\\
g_{21}&&&\\
.\\
.\\
g_{j-1,1}&&&g_{j-1,j}\\
g_{i,1}&&&g_{ij}\end{pmatrix}/\sigma_j=\frac {\langle
g\epsilon_ 1\wedge ..\wedge\epsilon_j,\epsilon_1\wedge
..\wedge\epsilon_{j-1} \wedge\epsilon_i\rangle}{\langle
g\epsilon_1\wedge ..\wedge\epsilon_ j,\epsilon_1\wedge
..\wedge\epsilon_j\rangle}\end{equation}
and for $i<j$,
$$u_{ij}=\det\begin{pmatrix} g_{11}&g_{12}&..&g_{1,i-1}&g_{1,j}\\
.&&&&g_{2,j}\\
.\\
g_{i,1}&&&&g_{i,j}\end{pmatrix}/\sigma_i.$$ Apply this to
$g=g_2^{-1}$ in a highest weight representation. Then
(\ref{lower}), together with (II.1), implies the claim.

The factorization (\ref{unobstructed}) is unobstructed. Thus it
exists. We can now read the calculation backwards, as in
(\ref{backwards}), and obtain a triangular factorization for $g_2$
as in (II.3) (initially for the restriction to a small circle
about $0$; but because $g_2$ is of finite type, this is valid also
for the standard circle). This completes the proof.
\end{proof}

In the $C^{\infty}$ analogue of Theorem \ref{Gtheorem1}, it is
necessary to add further hypotheses in parts I.1 and II.1; see
(\ref{highestweight}). To reiterate, we are now assuming that the
sequence $\{r_j\}_{j=1}^\infty$ is affine periodic.

\begin{theorem}\label{Gtheorem1smooth} Suppose that $\widetilde g_1 \in \widetilde L \dot{G}_0$
and $\Pi(\widetilde g_1)=g_1$. Consider the following three statements:
\begin{enumerate}
\item[(I.1)] $m(\widetilde g_1)=1$, and for each complex irreducible
representation $V(\pi)$ for $\dot{G}$, with lowest weight vector
$\phi \in V(\pi)$, $\pi(g_1)^{-1}(\phi)$ has holomorphic extension
to $\Delta$, is nonzero at all $z\in \Delta$, and is a positive
multiple of $v$ at $z=0$.
\item[(I.2)] $\widetilde g_1$ has a factorization of the form
$$\widetilde g_1=\lim_{n\to\infty}i_{\tau'_n}(g(\eta_n))..i_{\tau'_{-N}}(g(\eta_{-N})),$$
where $g(\eta_j)=k(\eta_j)$ for some $\eta_j\in\C$ (resp. $g(\eta_j)=q(\eta_j)$ for some $\eta_j\in\Delta$) when $\tau_j$ is a compact type (resp. non-compact type) root and the sequence $(\eta_j)_{j=-N}^\infty$ is rapidly decreasing.
\item[(I.3)] $\widetilde g_1$ has triangular factorization of the form
$\widetilde g_1=l_1a_1u_1$ where $l_1\in H^0(\Delta^*,\dot N^-)$
has smooth boundary values.
\end{enumerate}
Then statements (I.1) and (I.3) are equivalent. (I.2) implies (I.1) and (I.3).

Moreover, in the notation of (I.2),
\begin{equation}\label{product01}a_1=\prod_{j=-N}^{\infty}\mathbf a(\eta_j)^{h_{\tau^{\prime}_j}}.\end{equation}

Similarly, suppose that $\widetilde{g}_2 \in \widetilde{L}\dot{G}_0$
and $\Pi(\widetilde{g}_2)=g_2$. Consider the following three statements:
\begin{enumerate}
\item[(II.1)] $m(\widetilde g_2)=1$; and for each complex irreducible
representation $V(\pi)$ for $\dot{G}$, with highest weight vector
$v \in V(\pi)$, $\pi(g_2)^{-1}(v)\in H^0(\Delta;V)$ has
holomorphic extension to $\Delta$, is nonzero at all $z\in
\Delta$, and is a positive multiple of $v$ at $z=0$.

\item[(II.2)] $\widetilde g_2$ has a factorization of the form
\[
\widetilde g_2=\lim_{n\to\infty}i_{\tau_n}(g(\zeta_n))..i_{\tau_1}(g(\zeta_1))
\]
where $g(\zeta_j)=k(\zeta_j)$ for some $\zeta_j\in\C$ (resp. $g(\zeta_j)=q(\zeta_j)$ for some $\zeta_j\in\Delta$) when $\tau_j$ is a compact type (resp. non-compact type) root and the sequence $(\zeta_j)_{j=1}^\infty$ is rapidly decreasing.
\item[(II.3)] $\widetilde g_2$ has triangular factorization of the form
$\widetilde g_2=l_2a_2u_2$, where $l_2 \in
H^0(\Delta^*,\infty;\dot N^+,1)$ has smooth boundary values.
\end{enumerate}
Then statements (II.1) and (II.3) are equivalent. (II.2) implies (II.1) and (II.3).

Also, in the notation of (II.2),
\begin{equation}\label{product2}
a_2=\prod_{j=1}^{\infty}\mathbf a(\zeta_j)^{h_{\tau_j}}.
\end{equation}
\end{theorem}

\begin{conjecture}\label{conjecture2}
If $g_1$ is in the
identity connected component of the sets in (I.1) and (I.3), then the converse holds, i.e. $g_1$ has a root subgroup
factorization as in (I.2).  If $g_2$ is in the
identity connected component of the sets in (II.1) and (II.3), then the converse holds, i.e. $g_2$ has a root subgroup
factorization as in (II.2).
\end{conjecture}

In Remark \ref{proofofconjecture2}, at the end of the following proof, we will indicate how we envision proving this conjecture.
The issue in this $C^{\infty}$ context involves analysis, and we are not as confident in the truth of this Conjecture \ref{conjecture2}.

\begin{proof} The two sets of equivalences and implications are proven in the same way.
We consider the second set.

Suppose that (II.1) holds. To show that (II.3) holds, it suffices
to prove that $g_2$ has a triangular factorization with $l_2$ of
the prescribed form (see Lemma \ref{notationlemma}). By working in
a fixed faithful highest weight representation for $\dot{\mathfrak
g}$, without loss of generality, we can suppose $\dot G_0$ is a matrix subgroup
of $\SL(n,\mathbb C)$ (where $\dot {\mathfrak n}_+$ consists of upper triangular matrices). We will assume that this representation is
the complexified adjoint representation, or some subrepresentation of the exterior algebra of the adjoint representation, so
that we can suppose that $\dot G_0$ fixes a (indefinite) Hermitian form (in the case of the adjoint representation, this
is derived from the Killing form).

For the purposes of this proof, we will use the terminology in
Section 1 of \cite{P3}. We view $g_2\in L\G_0$ as a
multiplication operator on the Hilbert space $\mathcal
H=L^2(S^1;\mathbb C^n)$, and we write
$$M_{g_2}=\begin{pmatrix}A(g_2)&B(g_2)\\C(g_2)&D(g_2)\end{pmatrix}$$
relative to the Hardy polarization $\mathcal H=\mathcal H^+\oplus
\mathcal H^-$, where $A(g_2)$ is the compression of $M_{g_2}$ to
$\mathcal H^+$, the subspace of functions in $\mathcal H$ with
holomorphic extension to $\Delta$. To show that $g_2$ has a
Birkhoff factorization, we must show that $A(g_2)$ is invertible
(see Theorem 1.1 of \cite{P3}).

Let $C_1,..,C_n$ denote the columns of $g_2^{-1}$, and let $C_1^*,..,C_n^*$ denote the rows
of $g_2$. We can regard these as dual bases with respect to the pairing given by matrix multiplication, i.e.,
$C_i^*C_j=\delta_{ij}$.

The hypothesis of
(II.1) implies that both $C_1$ and $C_n^*$ have holomorphic extensions to $\Delta$ (in the latter case, by considering
the dual representation). Now suppose that $f\in \mathcal H^+$ is
in the kernel of $A(g_2)$. Then
\begin{equation}\label{keyfact}(C_j^*f)_+=0, \quad j=1,..,n,\end{equation}
where $(\cdot)_+$ denotes projection to $\mathcal H^+$. Since $C_n^*$
has holomorphic extension to $\Delta$, $(C_n^*f)_+=C_n^*f$ and therefore $C_n^*f$ is
identically zero on $S^1$ by (\ref{keyfact}). This implies that for $z\in S^1$,
$f(z)$ is a linear combination of the $n-1$ columns $C_j(z)$,
$j<n$. We write
\[
f=\lambda_1C_1+..+\lambda_{n-1} C_{n-1}
\]
where the coefficients are functions on the circle (defined a.e.).
Now consider the pointwise wedge product of $\mathbb C^n$ vectors
\[
f\wedge C_1\wedge..\wedge C_{n-2}=\pm
\lambda_{n-1}C_1\wedge..\wedge C_{n-1}.
\]
The vectors $C_1\wedge..\wedge C_{j}$ extend holomorphically to $\Delta$, and
never vanish, for any $j$, by (II.1) (by considering the
representation $\bigwedge^j(\mathbb C^n)$). Since $f$ also extends
holomorphically, this implies that $\lambda_{n-1}$ has holomorphic
extension to $\Delta$. Now
\[
C_{n-1}^*f=\lambda_{n-1}C_{n-1}^*C_{n-1}=\lambda_{n-1}
\]
by (\ref{keyfact}) and duality.

Since the right hand side
is holomorphic in $\Delta$, by (\ref{keyfact}) (for $j=n-1$)
$\lambda_{n-1}$ vanishes identically. This implies that in fact
$f$ is a (pointwise) linear combination of the first $n-2$ columns
of $g_2^{-1}$. Continuing the argument in the obvious way (by next
wedging $f$ with $C_1\wedge..\wedge C_{n-3}$ to conclude that
$\lambda_{n-2}$ must vanish), we conclude that $f$ is zero. This
implies that $\ker(A(g_2))=0$. Since $\dot G$ is simply connected,
$A(g_2)$ has index zero. Hence $A(g_2)$ is invertible. This
implies (II.3).

It is obvious that (II.3) implies (II.1); see
(\ref{highestweight}). Thus (II.1) and (II.3) are equivalent.

Before showing that (II.2) implies (II.1) and (II.3), we
need to explain why the $C^{\infty}$ limit in (II.2) exists. Because
$g(\zeta_j)=1+O(\vert\zeta_j\vert)$ as $\zeta_j \to 0$, the
condition for the product in (II.2) to converge absolutely is that
$\sum \zeta_n$ converges absolutely. So $g_2$ certainly represents
a continuous loop. The harder task is to check smoothness.

We will now calculate the derivative formally. In this
calculation, we let $g_2^{(n)}$ denote the product up to $n$, and
$\tau_n=q(n)d^*-\dot{\alpha}(n)$ ($q(n)>0$, and
$\dot{\alpha}(n)>0$). Then
\begin{equation}\label{formal}
g_2^{-1}(\frac{\partial g_2}{\partial\theta})
=\Pi(\sum_{n=1}^{\infty}Ad(g^{(n-1)})^{-1}\left(\iota_{\tau_n}(g(\zeta_n))^{-1}
\frac{\partial}{\partial\theta}\iota_{\tau_n}(g(\zeta_n))\right))
\end{equation}
$$=\sum_{n=1}^{\infty}Ad(g^{(n-1)})^{-1}\left(\sqrt{-1}
\frac{q(n)}{1\pm\vert\zeta_n\vert^2}(\mp \vert\zeta_n\vert^2h_{\dot{\alpha}(n)}-
\zeta_n e_{\dot{\alpha}(n)}z^{-q(n)}\mp
\bar{\zeta}_nf_{\dot{\alpha}(n)}z^{q(n)})\right).$$
Because we are
using an affine periodic sequence of simple reflections (with
period $w_l^{-1}\in C\subset \dot h_{\mathbb R}$),
$\tau_{l+1}=w_l^{-1}\cdot \tau_1$, $\tau_{l+2}=w_l^{-1}\tau_2$,
and so on. In general, writing $\tau_j=k(j)d^*-\dot{\alpha}(j)$ as
above, and using Proposition (4.9.5) of \cite{PS} to calculate the
coadjoint action,
\begin{equation}\label{taueqn}\tau_{nl+j}=w_l^{-n}\cdot
\tau_j=(q(j)+n\dot{\alpha}(j)(w_l))d^*-\dot{\alpha}(j).\end{equation}
Because $\dot{\alpha}(w_l)>0$, for all $\dot{\alpha}>0$, it
follows that $q(n)$ is asymptotically $n$. Because $Ad(g^{(n-1)})$
is orthogonal, (\ref{formal}) implies that
\[
\int\vert g_2^{-1}(\frac{\partial g_2}{\partial\theta})\vert^2
d\theta \le
\sum_{n=1}^{\infty}\|\mathrm{Ad}(g^{(n-1)})\|_2^2\frac{q(n)^2}{(1\pm\vert\zeta_n\vert^2)^2}(\vert\zeta_n\vert^4+\vert\zeta_n\vert^2)\vert
h_{\dot{\alpha}(n)}\vert^2
\]
by Bessel's inequality.  This is comparable to $\sum_{n=1}^\infty
n^2\vert\zeta_n\vert^2$ because $\|\mathrm{Ad}(g^{(n-1)})\|_2^2$ is uniformly bounded in $n$. Thus $g_2$ is $W^{1}$ (the $L^2$ Sobolev
space) whenever $(\zeta_j)\in w^1$. Higher derivatives can be
similarly calculated. This shows that if $\zeta\in w^n$, then
$g_2\in W^n$. Hence if $\zeta\in c^{\infty}$, the Frechet space of
rapidly decreasing sequences, then $g_2\in C^{\infty}$.

Now suppose that (II.2) holds. The map from $\zeta$ to $\widetilde
g_2$ is continuous, with respect to the standard Frechet
topologies for rapidly decreasing sequences and smooth functions.
The product (\ref{product2}) is also a continuous function of
$\zeta$, and hence is nonzero. This implies that $\widetilde g_2$
has a triangular factorization which is the limit of the
triangular factorizations of the finite products $\widetilde
g_2^{(n)}$. By Theorem \ref{Gtheorem1} and continuity, this
factorization will have the special form in (II.3). Thus (II.2)
implies (II.1) and (II.3).
\end{proof}

\begin{remark}\label{proofofconjecture2} We now want to give a naive argument for Conjecture \ref{conjecture2}.
Suppose that we are given $g_2$ as in (II.1) and (II.3). Recall
that $l_2$ has values in $\dot N^+$. We can therefore write
\begin{equation}\label{x*defn}l_2=\exp(\sum_{j=1}^{\infty}x_j^*f_{\tau_j}),\quad
x_j^*\in\mathbb C.\end{equation} (the use of $x^*$ for the
coefficients is consistent with our notation in the $SU(1,1)$ case,
see (II.3) of Theorem \ref{SU(1,1)theorem1}).

As a temporary notation, let $X$ denote the set of $g_2$ as in
(II.1) and (II.3); $x^*$ is a global linear coordinate for this
space. We consider the map
\begin{equation}\label{maps}c^{\infty}
\to X\text{ given by }\zeta \mapsto g_2.\end{equation} This map induces bijective
correspondences among finite sequences $\zeta$, $g_2\in X\cap
L_{fin}\dot K$ and finite sequences $x^*$, and the maps $\zeta$ to
$x^*$ and $x^*\to \zeta$ are given by rational maps (i.e. rational in $\zeta$ and $\bar{\zeta}$); however
(although it seems likely) it is not known that the limits of
these rational maps actually make sense even for rapidly
decreasing sequences (see the Appendix of \cite{P3} for the
$SU(2)$ case). We will use an inverse function argument to show
that the map (\ref{maps}) has a global inverse (technically, to
apply the inverse function theorem, we should consider the maps of
Sobolev spaces $w^n \to X^n$, where $X^n$ is the $W^n$ completion
of $X$, but we will suppress this).

Given a variation of $\zeta$, denoted $\zeta^{\prime}$, we can
formally calculate the derivative of this map,
\begin{equation}\label{derivativeform}g_2^{-1}g_2'=\sum_{n=1}Ad(g_2^{(n-1)})^{-1}(i_{\tau_n}(
\mathbf a(\zeta_n)\begin{pmatrix} 1&\bar{\zeta}_n\\
-\zeta_n&1\end{pmatrix}\{\mathbf
a(\zeta_n)'\begin{pmatrix} 1&\bar{\zeta}_
n\\
\zeta_n&1\end{pmatrix}+\mathbf
a(\zeta_n)\begin{pmatrix} 0&\bar{\zeta}_
n'\\
\zeta_n'&0\end{pmatrix}\}))$$
$$=\sum_{n=1}Ad(g_2^{(n-1)})^{-1}(i_{\tau_n}(
\mathbf a(\zeta_n)^{-1}\mathbf a(\zeta_n)'\begin{pmatrix} 1&0\\
0&1\end{pmatrix}+\mathbf a(\zeta_n)^2\begin{pmatrix}
\bar{\zeta}_n\zeta_
n'&\bar{\zeta}_n'\\
\zeta_n'&\zeta_n\zeta_n'\end{pmatrix}))\end{equation}
$$=\sum_{n=1}Ad(g_2^{(n-1)})^{-1}(i_{\tau_n}(\mathbf a(\zeta_n)^2\begin{pmatrix} \frac
12(\bar{\zeta}_n\zeta_n'-\zeta_n\bar{\zeta}_n')&\bar{\zeta}_n'\\
\zeta_n'&-\frac
12(\bar{\zeta}_n\zeta_n'-\zeta_n\bar{\zeta}_n')\end{pmatrix} ))$$ As before it is clear that this is convergent, so that
(\ref{maps}) is smooth. At $\zeta=0$ this is clearly injective
with closed image, so that there is a local inverse. Consider more
generally a fixed $g_2\in X \cap L_{fin}\dot G_0$, so that
$g_2^{(n-1)}=g_2$ for large $n$. Recall that the root spaces for
the $\tau_n$ are independent and fill out $\dot{\mathfrak
n}^{-}(z\mathbb C[z])$. Given a variation such that
$g_2^{-1}g_2^{\prime}=0$, the terms in the last sum in the
derivative formula (\ref{derivativeform}) must be zero for large
$n$. But we know that the map (\ref{maps}) is a bijection on
finite $\zeta$. Thus for a variation of a finite number of
$\zeta_j$ which maps to zero, the variation vanishes. It is clear
that the image of the derivative (\ref{derivativeform}) is closed.
The image is therefore the tangent space to $X$ (because we know
that finite variations will fill out a dense subspace of the
tangent space). This implies there is a local inverse. This local
inverse is determined by its values on finite $x^*$, and hence
there is a uniquely determined global inverse. This shows that
(II.1) and (II.3) imply (II.2).

Finally (\ref{product2}) follows by continuity from
(\ref{productformula}).
\end{remark}

\subsection{Generalization of Theorem
\ref{SU(1,1)theorem2}}\label{topstratum2}

\begin{theorem}\label{Gsmooththeorem}Suppose $\widetilde g \in \widetilde L\dot G_0$ and $\Pi(\widetilde g)=g$.
\begin{enumerate}
\item[(a)] The following are equivalent:
\begin{enumerate}
\item[(i)] $\widetilde g$ has a triangular factorization $\widetilde
g=lmau$, where $l$ and $u$ have $C^{\infty}$ boundary values, and satisfy the conditions
$l(z),u^{-1}(z)\in \dot G_0\dot B^+$ for all $z\in S^1$.
\item[(ii)] $\widetilde g$ has a (partial root subgroup) factorization of the form
$$\widetilde g=\Theta(\widetilde g_1^*)
\exp(\chi)\widetilde g_2,$$ where $\chi \in \widetilde L\dot{\mathfrak t}$,
and $\widetilde g_1$ and $\widetilde g_2$ are as in (I.3) and (II.3) of Theorem
\ref{Gtheorem1smooth}, respectively.
\end{enumerate}
\item[(b)] In reference to (ii) of part (a),
\begin{equation}\label{diagonalclaim}a(\widetilde g)=a(g)=
a(g_1)a(\exp(\chi))a(g_2),\quad
\Pi(a(g))=\Pi(a(g_1))\Pi(a(g_2))
\end{equation}
and
\begin{equation}\label{abeliandiagonal}
a(\exp(\chi))=\vert \sigma_0
\vert(\exp(\chi))^{h_0}\prod_{j=1}^r \vert \sigma_0
\vert(\exp(\chi))^{\check a_j h_j}.
\end{equation}
\end{enumerate}
\end{theorem}

\begin{remarks} Suppose that $\dot G_0=SU(1,1)$. In this case the last condition in (i) in Theorem
\ref{Gsmooththeorem}, that
$l(z),u^{-1}(z)\in \dot G_0\dot B^+$, is equivalent to the condition in Theorem \ref{SU(1,1)theorem2} that
the boundary values $l_{21}/l_{11}$
and $u_{21}/u_{22}$ are $<1$ in magnitude on $S^1$, and part (b) specializes to the statement of
Theorem \ref{Toeplitz}.
\end{remarks}

\begin{proof} Our strategy of proof is the following. We will first show that in
part (a), (ii) implies (i). In the process we will prove part (b). We will then
show that (i) implies (ii).

Suppose that we are given $\widetilde g$ as in (ii). Both $\widetilde g_1$ and $\widetilde g_2$
have triangular factorizations by Theorem \ref{Gtheorem1smooth}. In the notation of Theorem
\ref{Gtheorem1smooth},
\begin{equation}\label{theta_included}
\widetilde g=\Theta((l_1a_1u_1)^*)\exp(\chi)(l_2a_2u_2)=\Theta(u_1^*) a_1
(\Theta(l_1^*)\exp(\chi)l_2)a_2u_2
\end{equation}
since $\Theta$ preserves the $A$ factor.
The basic observation is that
\begin{equation}\label{bfactor}
b=\Theta(l_1^*)\exp(\chi)l_2\in (\widetilde L\dot B^+)_0
\end{equation}
(the inverse image in the affine extension for
the identity component of loops in $\dot B^+$), and $b$ will have a triangular factorization which
we can compute. To do this requires some care with the central extension, and this involves some preparation.

Because $\dot B^+$ is the semidirect product of $\dot H$ and $\dot N^+$, there is an isomorphism of loop groups
\[
L\dot B^+=L\dot H\ltimes L\dot N^+.
\]
The central extension is trivial for $L\dot N^+$,
and hence there is an isomorphism
\[
\widetilde L\dot B^+=\widetilde L\dot H\ltimes L\dot N^+
\]
where the action of $\widetilde L\dot H$ on $L\dot N^+$ is the same as the conjugation action of
$L\dot H$ on $L\dot N^+$, and $\widetilde L\dot H$ is a Heisenberg extension determined by the bracket
(\ref{bracket}).

Given $\chi \in \widetilde L\dot{\mathfrak t}$ as above, let
$\chi=\chi_-^*+\chi_0+\chi_+$ denote the linear triangular decomposition, where
$\chi_0\in
\mathfrak t$, $\chi_+\in H^0(\Delta,0;\dot{\mathfrak
h},0)$ and $\chi_-=-\chi_+^*$. Then (calculating in terms of the Heisenberg extension)
\begin{eqnarray*}
\exp(\chi)& = & \exp(\chi_-)\exp(\chi_0)\exp(-[\chi_-,\chi_+])\exp(\chi_+) \\
&= & \exp(\chi_-)\exp(\chi_0)\exp(\sum_{j=1}^{\infty}j\langle\chi_j,\chi_j\rangle c)\exp(\chi_+).
\end{eqnarray*}
Substituting this into (\ref{bfactor}) we find
\[
b=\exp(\chi_-)b_1\exp(\chi_+)
\]
where
\[
b_1=\exp(-\chi_-)\Theta(l_1^*)\exp(\chi_-)\exp(\chi_0)\exp(\sum_{j=1}^{\infty}j\langle\chi_j,\chi_j\rangle c)\exp(\chi_+)l_2\exp(-\chi_+).
\]
Thus, $b$ has a triangular factorization
\[
b=\left(\exp(\chi_-)L\right)\left(m(b)a(b)\right)\left(U\exp(\chi_+)\right),
\]
where $m(b)=m(b_1)=\exp(\chi_0)$, $a(b)=a(b_1)=\exp(\sum_{j=1}^{\infty}j\langle\chi_j,\chi_j\rangle c)$,
\[
L=l(\exp(-\chi_-)\Theta(l_1^*)\exp(\chi_-)\exp(\chi_0)\exp(\chi_+)l_2\exp(-\chi_+))\in H^0(\Delta^*,\infty;\dot
N^+,1),
\]
and
\[
U=u(\exp(-\chi_-)\Theta(l_1^*)\exp(\chi_-)\exp(\chi_0)\exp(\chi_+)l_2\exp(-\chi_+))\in H^0(\Delta;\dot
N^+).
\]
Thus, from (\ref{theta_included}), $\widetilde g$ will have a triangular factorization $l(\widetilde g)m(\widetilde g)a(\widetilde g)u(\widetilde g)$ with
\begin{equation}\label{gfactorization}
l(\widetilde g)=\Theta(u_1^*)\exp(\chi_-)a_1L a_1^{-1},\quad
m(\widetilde g)=m(b)=\exp(\chi_0),
\end{equation}
\[
a(\widetilde g)=a_1a_2\exp(\sum_{j=1}^{\infty}j\langle\chi_j,\chi_j\rangle c),\quad
u(\widetilde g)=a_2^{-1}Ua_2\exp(\chi_+)u_2.
\]
Thus, (ii) implies (i) in part (a).  At the same time this also implies part (b).

Now we need to show that (i) implies (ii). For this direction, there is not any need to consider the central extension, so we will no longer use tildes for group elements.

Suppose $g=lmau$, as in (i). At each point of the
circle there exist $\dot N^+ \dot A \dot G_0$ decompositions
\begin{equation}\label{questionable}
l^{-1}=\dot n_1\dot a_1 \dot g_1,\quad u=\dot n_2
\dot a_2 \dot g_2.
\end{equation}
This is a consequence of the somewhat bizarre hypotheses in (i).  Then $\dot g_1=\Theta(\dot g_1^{-1})^*=\dot a_1^{-1}\Theta (\dot n_1^*)\Theta(\ell^*)$ since $\dot g_2\mapsto \Theta(\dot g_2^{-1})^*$ is the involution fixing $\dot G_0$ in $\dot G$, and $\Theta$ acts as the inverse on $\dot A$ under the inner type assumption.

In turn, there are Birkhoff decompositions
\[
\dot a_i^{-1}=\exp(\chi_i^*+\chi_{i,0}+\chi_i),
\quad \chi_i\in H^0(\Delta,\dot \h),\quad \chi_{i,0}\in \dot
\h_{\mathbb R}
\]
for $i=1,2$. Define
\[
g_i=\exp(-\chi_i^*+\chi_i)\dot g_i
\]
for $i=1,2$. Then
\[
g_1=\exp(-\chi_{1,0}-2\chi_1^*)\Theta(\dot n_1^*)\Theta(\ell^*)
\]
has
triangular factorization with
\[
l(g_1)=l(\exp(-\chi_{1,0}-2\chi_1^*)\Theta(\dot n_1^*)\exp(\chi_{1,0}+2\chi_1^*))\in H^0(\Delta^*,\infty; \dot N^-,1),
\]
\[
\Pi(a(g_1))=\exp(\chi_{1,0}),
\]
and similarly
\[
g_2=\exp(2\chi_2+\chi_{2,0})\dot n_2^{-1}u
\]
has
triangular factorization with
\[
l(g_2)=l(
\exp(2\chi_2+\chi_{2,0})\dot n_2^{-1}\exp(-2\chi_2-\chi_{2,0}))\in
H^0(\Delta^*,\infty;\dot N^+,1),
\]
\[
\Pi(a(g_2))=\exp(\chi_{2,0}).
\]

The conclusion is somewhat miraculous.  On the one hand $\Theta(g_1^*)^{-1} g
g_2^{-1} $ has values in $\dot G_0$ because $g_1\mapsto \Theta(g_1^*)^{-1}$ is the pointwise involution fixing $\dot G_0$ in $\dot G$.  On the other hand
\begin{eqnarray}
\Theta(g_1^*)^{-1} g g_2^{-1} & = & \Theta(g_1^*)^{-1}lma (\exp(2\chi_2+\chi_{2,0})\dot n_2^{-1}u)^{-1} \nonumber \\
& = & \Theta (g_1^*)^{-1}lma\dot n_2\exp(-2\chi_2-\chi_{2,0}) \nonumber \\
& = & \exp(-\chi_{1,0}-2\chi_1^*)\dot n_1ma\dot n_2\exp(-2\chi_2-\chi_{2,0}) \label{uppert}
\end{eqnarray}
has values in $\dot B^+$. Therefore $\Theta(g_1^*)^{-1} g g_2^{-1}$ has values in $\dot G_0\cap\dot B^+=\dot T$. It is
also clear that (\ref{uppert}) is connected to the identity, and
hence $\Theta(g_1^*)^{-1} g g_2^{-1}\in (L\dot T)_0$ and thus equals $\exp(\chi)$. Hence, $g=\Theta(g_1^*)\exp(\chi)g_2$.  Thus (i) implies (ii).
\end{proof}

\end{document}